\pdfoutput=1
\RequirePackage{ifpdf}
\ifpdf 
\documentclass[pdftex]{sigma}
\else
\documentclass{sigma}
\fi

\newtheorem{Theorem}{Theorem}[section]

\newtheorem{Lemma}[Theorem]{Lemma}
\newtheorem{Proposition}[Theorem]{Proposition}
\newtheorem{Conjecture}[Theorem]{Conjecture}
 { \theoremstyle{definition}
\newtheorem{Definition}[Theorem]{Definition}

\newtheorem{Example}[Theorem]{Example}
\newtheorem{Remark}[Theorem]{Remark}
\newtheorem{Question}[Theorem]{Question}}

\newcommand{\bH}{{\mathbb{H}}}

\newcommand{\C}{{\mathbb{C}}}
\newcommand{\R}{{\mathbb{R}}}
\newcommand{\Z}{{\mathbb{Z}}}
\newcommand{\K}{{\mathbb{K}}}

\newcommand{\bP}{{\mathbb{P}}}
\newcommand{\bS}{{\mathbb{S}}}

\newcommand{\CP}{{\mathbb{CP}}}
\newcommand{\RP}{{\mathbb{RP}}}
\newcommand{\KP}{{\mathbb{KP}}}

\newcommand{\Oo}{{\mathcal O}}
\newcommand{\Xx}{{\mathcal X}}
\newcommand{\Dd}{{\mathcal D}}
\newcommand{\Pp}{{\mathcal P}}
\newcommand{\Uu}{{\mathcal U}}
\newcommand{\Tt}{{\mathcal T}}

\newcommand{\sG}{{\mathsf{G}}}
\newcommand{\sH}{{\mathsf{H}}}
\newcommand{\sP}{{\mathsf{P}}}
\newcommand{\sB}{{\mathsf{B}}}
\newcommand{\sPGL}{{\mathsf{PGL}}}
\newcommand{\sGL}{{\mathsf{GL}}}
\newcommand{\sSL}{{\mathsf{SL}}}
\newcommand{\sPSL}{{\mathsf{PSL}}}
\newcommand{\sPO}{{\mathsf{PO}}}
\newcommand{\sO}{{\mathsf{O}}}
\newcommand{\sSO}{{\mathsf{SO}}}
\newcommand{\sPSp}{{\mathsf{PSp}}}
\newcommand{\sSp}{{\mathsf{Sp}}}
\newcommand{\sU}{{\mathsf{U}}}

\newcommand{\fg}{{\mathfrak g}}
\newcommand{\defin}[1]{\emph{#1}}
\newcommand{\ra}{\rightarrow}
\newcommand{\lra}{{\longrightarrow}}

\newcommand{\Hol}{\mathrm{Hol}}
\newcommand{\Id}{\mathrm{Id}}

\DeclareMathOperator{\Hom}{Hom}
\DeclareMathOperator{\Aut}{Aut}
\DeclareMathOperator{\Aff}{Aff}
\DeclareMathOperator{\End}{End}
\DeclareMathOperator{\Iso}{Iso}

\DeclareMathOperator{\tr}{tr}
\DeclareMathOperator{\Span}{Span}
\DeclareMathOperator{\Symm}{Symm}

\DeclareMathOperator{\Stab}{Stab}
\DeclareMathOperator{\Vol}{Vol}

\DeclareMathOperator{\Real}{Re}

\DeclareMathOperator{\Equiv}{Equiv}
\DeclareMathOperator{\Diffeo}{Diffeo}
\DeclareMathOperator{\Dev}{Dev}
\DeclareMathOperator{\Flat}{Flat}
\DeclareMathOperator{\Isom}{Isom}
\DeclareMathOperator{\Gr}{Gr}
\DeclareMathOperator{\Hit}{Hit}
\DeclareMathOperator{\QHit}{QHit}

\DeclareMathOperator{\QFuch}{QFuch}
\DeclareMathOperator{\Anosov}{Anosov}

\begin{document}

\allowdisplaybreaks

\newcommand{\arXivNumber}{1809.07290}

\renewcommand{\thefootnote}{}

\renewcommand{\PaperNumber}{039}

\FirstPageHeading

\ShortArticleName{Higgs Bundles and Geometric Structures on Manifolds}

\ArticleName{Higgs Bundles and Geometric Structures on Manifolds\footnote{This paper is a~contribution to the Special Issue on Geometry and Physics of Hitchin Systems. The full collection is available at \href{https://www.emis.de/journals/SIGMA/hitchin-systems.html}{https://www.emis.de/journals/SIGMA/hitchin-systems.html}}}

\Author{Daniele ALESSANDRINI}

\AuthorNameForHeading{D.~Alessandrini}

\Address{Ruprecht-Karls-Universitaet Heidelberg, INF 205, 69120, Heidelberg, Germany}
\Email{\href{mailto:daniele.alessandrini@gmail.com}{daniele.alessandrini@gmail.com}}
\URLaddress{\url{https://www.mathi.uni-heidelberg.de/~alessandrini/}}

\ArticleDates{Received September 28, 2018, in final form April 17, 2019; Published online May 10, 2019}

\Abstract{Geometric structures on manifolds became popular when Thurston used them in his work on the geometrization conjecture. They were studied by many people and they play an important role in higher Teichm\"uller theory. Geometric structures on a manifold are closely related with representations of the fundamental group and with flat bundles. Higgs bundles can be very useful in describing flat bundles explicitly, via solutions of Hitchin's equations. Baraglia has shown in his Ph.D.~Thesis that Higgs bundles can also be used to construct geometric structures in some interesting cases. In this paper, we will explain the main ideas behind this theory and we will survey some recent results in this direction, which are joint work with Qiongling Li.}


\Classification{57M50; 53C07; 22E40}

\renewcommand{\thefootnote}{\arabic{footnote}}
\setcounter{footnote}{0}

\section{Introduction}

The theory of geometric structures on manifolds was introduced by Cartan and Ehresmann in the 1920s, following the ideas given by Klein in his Erlangen program. This theory became popular in the 1980s, when Thurston used it in the statement of his Geometrization Conjecture. Since then, many people contributed important results, see for example \cite{BabaGrafting1,BabaGrafting2,BauesTorus,ChoiGoldmanRP2,ChoiGoldmanClassification,GalloKapovichMarden, GoldmanFuchsianHol,GoldmanConvex}.

Nowadays, geometric structures on manifols are also important in higher Teichm\"uller theo\-ry, a research area that arose from the work of Goldman \cite{GoldmanConvex}, Choi--Goldman \cite{ChoiGoldmanRP2}, Hitchin~\cite{liegroupsteichmuller}, Labourie~\cite{AnosovFlowsLabourie}, Fock--Goncharov \cite{FockGoncharov}. They studied some connected components of the character varieties of surface groups in higher rank Lie groups which share many properties with Teichm\"uller spaces. They are now called Hitchin components or higher Teichm\"uller spaces.

The first works which related geometric structures and higher Teichm\"uller theory are Choi--Goldman \cite{ChoiGoldmanRP2} and Guichard--Wienhard \cite{convexfoliatedprojective}, showing how the low-rank Hitchin components can be used as parameter spaces of special geometric structures on closed manifolds. These first results were then generalized by Guichard--Wienhard \cite{GWDomainsofDiscont} and Kapovich--Leeb--Porti \cite{KLPAnosov1}, who show that Anosov representations can often be used to construct geometric structures on closed manifolds.

Higgs bundles are an important tool in higher Teichm\"uller theory because they can be used to describe the topology of the character varieties (see for example Hitchin \cite{selfduality,liegroupsteichmuller}, Alessandrini--Collier \cite{SO23LabourieConj}). Anyway, they were initially believed to give very little information on the geometry of a single representation. This point of view is changing, since we have now many examples where the Higgs bundle can be used to give interesting information on the geometric structures associated with a certain representation of a surface group.

The main purpose of this survey paper is to explain these constructions. The first ones were presented by Baraglia in his Ph.D.~Thesis~\cite{BaragliaThesis}, more recent ones are in Alessandrini--Li \cite{AdSpaper,NilpotentCone,ProjectiveStructuresHB}, and Collier--Tholozan--Toulisse \cite{CollierTholozanToulisse}. The main idea behind these constructions is that a geometric structure corresponds to a section of a flat bundle which is transverse to the parallel foliation. The holomorphic structure of the Higgs bundle helps to construct sections, and the parallel foliation can be described by solving Hitchin's equations.

I will initially describe the fundamental notions of the theory of geometric structures on manifolds: the notion of geometry in the sense of Klein, geometric manifolds, the relationship with the theory of domains of discontinuity for Anosov representations, the relationship with representations of fundamental groups of manifolds, and the deformation spaces of geometric structures on a fixed topological manifold, see Section \ref{sec:geometric structures}.

Then I will give an introduction to character varieties and their relationship with the moduli space of flat bundles. I will also introduce the subspaces of the character varieties that are most important in higher Teichm\"uller theory, see Section~\ref{sec:character varieties}.

After this, everything is ready to explain the relationship between geometric structures and flat bundles, via a tool called the graph of a geometric structure, see Section \ref{sec:graph}.

We will then enter in the main topic of the mini-course, using Higgs bundles and solutions of Hitchin's equations to describe flat bundles explicitly and construct geometric structures. In Section~\ref{sec:higgs bundles} three simple examples are given where this method works and allows us to construct hyperbolic structures, complex projective structures and convex real projective structures on surfaces.

Higgs bundles can only describe flat bundles on surfaces, but we also want to describe flat bundles on higher-dimensional manifolds. See Section \ref{sec:higher dimension} for an explanation on how flat bundles on manifolds of different dimension can be related, and an exposition of interesting open problems in the theory of geometric manifolds that are related with this issue.

We will finally see how geometric structures on higher-dimensional manifolds can be constructed. As a warm-up, in Section \ref{sec:circle bundles} we consider the case of $3$-dimensional manifolds, and we see how to construct the convex foliated real projective structures and the anti-de Sitter structures on circle bundles over surfaces.

In the last part, in Section \ref{sec:higher dimensions}, we will see how the technique works in the case of manifolds of higher dimension. In this final case, the technical details are more involved and will be mainly left out. We will see how to construct real and complex projective structures on higher-dimensional manifolds, and how this result has applications to the theory of domains of discontinuity for Anosov representations.

This survey paper is based on the lecture notes for the mini-course ``Higgs bundles and geometric structures on manifolds'' that I gave at the University of Illinois at Chicago during the program ``Workshop on the Geometry and Physics of Higgs bundles II'', November 11--12, 2017. The mini-course was targeted at graduate students and young post-docs with an interest in Gauge Theory and Higgs bundles and this survey paper addresses the same public.

\section{Geometric structures on manifolds} \label{sec:geometric structures}

This section will be an introduction to the theory of geometric structures on manifolds, their developing maps and holonomy representations. For more details about this theory, see Thurston's book \cite{ThurstonBook} or Goldman's notes \cite{GoldmanGSAVOR}.

\subsection{Geometries}

The theory of geometric structures on manifolds traces its origins back to Felix Klein who, in his Erlangen program (1872) discussed what is geometry. Klein's idea is that geometry is the study of the properties of a space that are invariant under the action of a certain group of symmetries. The main examples he had in mind were the Euclidean geometry, where the space is $\R^n$ and the group is $\Isom\big(\R^n\big)$, and the affine geometry, where the space is $\R^n$ and the group is $\Aff\big(\R^n\big)$. These geometries study exactly the same space, but they focus on very different properties. Euclidean geometry deals with lengths, angles, and circles, the notions that are invariant under the group of isometries. These notions make no sense in affine geometry, because they are not preserved by the affine group. Affine geometry, instead, deals with ratios of lengths, parallelism and ellipses. Klein emphasizes that when studying geometry, the symmetry group is as important as the space. Let's now give a definition in modern terms.

\begin{Definition}A \defin{geometry} is a pair $(X,\sG)$, where $\sG$, the \defin{symmetry group} is a Lie group and~$X$, the \defin{model space}, is a manifold endowed with a transitive and effective action of $\sG$. Recall that an action is \defin{effective} if every $g\in \sG {\setminus} \{e\}$ acts non-trivially on~$X$.

If $U \subset X$ is an open subset, we will say that a map $f\colon U\ra X$ is \defin{locally in $\sG$} if for every connected component $C$ of $U$, there exists $g\in \sG$ such that $f|_C = g|_C$.

For $x\in X$, the \defin{isotropy group} of $x$ in $\sG$ is the subgroup
\begin{gather*}\sH = \Stab_\sG(x) = \{h\in \sG \,|\, h(x) = x\}.\end{gather*}
\end{Definition}

The isotropy group $\sH$ is a closed subgroup of $\sG$. Since the action is transitive, the conjugacy class of the isotropy group does not depend on the choice of the point~$x$.

As an equivalent definition, a geometry can be defined as a pair $(\sG,\sH)$, where $\sG$ is a Lie group and $\sH$ is a closed subgroup of~$\sG$, up to conjugation. The model space can then be reconstructed as the quotient $X=\sG/\sH$. From this description, we see that $X$ inherits from $\sG$ a structure of real analytic manifold such that the action of~$\sG$ on~$X$ is real analytic.

\begin{Example}Classical examples of geometries are the \defin{Euclidean geometry} $\big(\R^n, \Isom\big(\R^n\big)\big)$, the \defin{affine geometry} $\big(\R^n, \Aff\big(\R^n\big)\big)$ and the \defin{real projective geometry} $\big(\RP^n, \sPGL(n+1,\R)\big)$. There are many other examples which we will organize in families.
\begin{enumerate}\itemsep=0pt
\item A geometry is said to be \defin{of Riemannian type} if $\sG$ acts on $X$ preserving a Riemannian metric. This happens if and only if the isotropy group is compact. Examples are the \defin{isotropic geometries}
(the Euclidean geometry $\big(\R^n, \Isom\big(\R^n\big)\big)$, the \defin{hyperbolic geometry} $\big(\bH^n, \sPO(1,n)\big)$ and the \defin{spherical geometry} $\big(\bS^n, \sPO(n+1)\big)$), the geometries of symmetric spaces and the geometries of Lie groups ($(\sG,\sG)$, where~$G$ acts on itself on the left). Thurston's eight $3$-dimensional geometries~\cite{ThurstonBook} are in this family.
\item A geometry is said to be of \defin{of pseudo-Riemannian type} if $\sG$ acts on $X$ preserving a pseudo-Riemannian metric. Examples are many geometries coming from the theory of relativity, such as the geometry of Minkowski space $\big(\R^n, \sO(1,n-1)\ltimes\R^n\big)$, of the anti-de Sitter space $\big({\rm AdS}^n, \sPO(2,n-1)\big)$, and of the de Sitter space $\big({\rm dS}^n, \sPO(1,n)\big)$.
\item A geometry is said to be \defin{of parabolic type} if the isotropy group $\sH$ is a parabolic subgroup of $\sG$. Examples are the real projective geometry $\big(\RP^n, \sPGL(n+1,\R)\big)$, the \defin{complex projective geometry} $\big(\CP^n, \sPGL(n+1,\C)\big)$, the \defin{conformal geometry} $\big(\bS^n, \sPO(1,n+1)\big)$, the geometry of Grassmannians and of Flag manifolds.
\end{enumerate}
\end{Example}

\begin{Remark}The notation $(X,\sG)$ for a geometry is, in most practical cases, too cumbersome, hence we will usually denote the geometry just by~$X$, when this does not result in ambiguities. For example, we will often denote the real projective geometry by $\RP^n$, instead of $\big(\RP^n, \sPGL(n+1,\R)\big)$. Similarly for $\CP^n$, $\bH^n$, ${\rm AdS}^n$, these symbols will denote both the model space and the geometry.
\end{Remark}

\subsection{Geometric manifolds}

Every geometry can be used as a local model for geometric structures on manifolds. This idea was introduced by Cartan and Ehresmann in the 1920s, and it was made popular by Thurston around 1980, when he used it in the statement of his geometrization conjecture (now Perelman's theorem).

\begin{Definition}Given a geometry $(X,G)$ and a manifold $M$ with $\dim(M)=\dim(X)$, an \defin{$(X,G)$-structure} on $M$ is a maximal atlas $\mathcal{U} = \{(U_i,\varphi_i)\}$ where
\begin{enumerate}\itemsep=0pt
\item[1)] $\{U_i\}$ is an open cover of $M$,
\item[2)] the functions
\begin{gather*}\varphi_i\colon \ U_i \ra X\end{gather*}
are homeomorphisms with the image, which is an open subset of $X$,
\item[3)] the transition functions
\begin{gather*}\varphi_i\circ \varphi_j^{-1}\colon \ \varphi_j(U_i \cap U_j) \ra \varphi_i(U_i \cap U_j) \end{gather*}
are locally in $\sG$.
\end{enumerate}
An \defin{$(X,\sG)$-manifold} is a manifold endowed with an $(X,\sG)$-structure.
\end{Definition}

An $(X,\sG)$-manifold is a real analytic manifold, because the transition functions of the atlas are real analytic. Moreover, on an $(X,\sG)$-manifold $M$, all the local properties of $X$ that are preserved by $\sG$ are given to~$M$ by the atlas. For example, if $(X,\sG)$ is of (pseudo-)Riemannian type, every $(X,\sG)$-manifold inherits a (pseudo-)Riemannian metric from~$X$. Similarly, every manifold with a real or complex projective structure has a well defined notion of projective line: some real or complex $1$-dimensional submanifold that is mapped to a projective line by any chart. Moreover, given $4$ points on such a projective line, it is possible to compute their cross-ratio.

\begin{Example} \label{exa:geometric manifolds}\quad
\begin{enumerate}\itemsep=0pt
\item For every geometry $(X,G)$, take $M=X$. The identity map is a global chart for the tautological $(X,G)$-structure on~$M$. Slightly more generally, if $M \subset X$ is an open subset, again the identity map is a global chart for an $(X,G)$-structure on~$M$.
\item Consider the Euclidean geometry: $\big(\R^n, \Isom\big(\R^n\big)\big)$. Let $M$ be the torus
\begin{gather*}M = T^n = \R^n/\Z^n.\end{gather*}
We can construct an atlas using the covering $\R^n\ra M$: every well covered open set is one of the $U_i$s, and every section of the covering over such a $U_i$ is one of the $\varphi_i$s.
\item Consider the hyperbolic geometry: $\big(\bH^2, \sPO(1,2)\big)$. Let $M$ be a closed surface of genus $g\geq 2$. Recall that such a surface can be obtained by gluing the sides of a $(4g)$-gon along the standard pattern $a,b,a^{-1},b^{-1}, c,d, c^{-1},d^{-1}, \dots$, obtaining a cell complex with $1$ vertex, $2g$ edges and $1$ face. To put an $\bH^2$-structure on~$M$, we first need to construct a regular $(4g)$-gon in $\bH^2$ such that the sum of the internal angles of the polygon is $\frac{2\pi}{4g}$. When the edges of the polygon are glued with the standard pattern, they give a surface of genus $g$ with an $\bH^2$-structure.
\item Consider the real projective geometry: $\big(\RP^2, \sPGL(3,\R)\big)$. The subgroup $\sPO(1,2) < \sPGL(3,\R)$ acts on $\RP^2$ preserving a disc, this is the \defin{Klein model} of the hyperbolic plane: there is a $\sPO(1,2)$-equivariant map $K\colon \bH^2 \ra \RP^2$ with image this disc. The $\bH^2$-structure on $M$ constructed in point (3) induces an $\RP^2$-structure by composing the charts with the map $K$.
\item Consider the complex projective geometry: $\big(\CP^1, \sPGL(2,\C)\big)$. The subgroup $\sPSL(2,\R) < \sPGL(2,\C)$ acts on $\CP^1$ preserving the upper half plane, this is the \defin{Poincar\'e model} of the hyperbolic plane: the connected component $\sPO_0(2,1)$ of $\sPO(2,1)$ is isomorphic to $\sPSL(2,\R)$ in such a way that there is a $\sPO_0(2,1)$-equivariant map $P\colon \bH^2 \ra \CP^1$ with image the upper half plane. The $\bH^2$-structure on $M$ constructed in point (3) induces a~$\CP^1$-structure by composing the charts with the map $P$.b
\end{enumerate}
\end{Example}

\subsection{Morphisms} \label{subsec:morphisms}

\begin{Definition}Given two $(X,\sG)$-manifolds $M$, $N$, a map $f\colon M\ra N$ is an \defin{$(X,\sG)$-map} if for every $m\in M$, there exist charts $(U,\varphi)$ for $M$ around $m$ and $(V,\psi)$ for $N$ around $f(m)$ such that $f(U) \subset V$ and the composition
\begin{gather*}\psi \circ f \circ \varphi^{-1} \colon \ \varphi(U) \ra \psi(V)\end{gather*}
is locally in $\sG$.
\end{Definition}

The $(X,\sG)$-maps are always real analytic local diffeomorphisms. Composition of $(X,\sG)$-maps is an $(X,\sG)$-map, hence we can form a category having the $(X,\sG)$-manifolds as objects and the $(X,\sG)$-maps as arrows.

\begin{Definition}
An \defin{$(X,\sG)$-isomorphism} is a diffeomorhism which is also an $(X,\sG)$-map. An \defin{$(X,\sG)$-automorphism} is an isomorphism between an $(X,\sG)$-manifold and itself.
\end{Definition}

Notice that the inverse of an $(X,G)$-isomorphism is automatically an $(X,\sG)$-map. If $M$ is an $(X,\sG)$-manifold, we will denote its group of automorphisms by
\begin{gather*}\Aut_{(X,\sG)}(M) = \{f\colon M\ra M \,|\, f \text{ is an } (X,\sG)\text{-isomorphism}\}.\end{gather*}
These groups can sometimes be understood:

\begin{Proposition}
\begin{gather*}\Aut_{(X,\sG)}(X) = \sG.\end{gather*}
More generally, if $U\subset X$ is open, then
\begin{gather*}\Aut_{(X,\sG)}(U) = \{g\in \sG \,|\, g(U) = U\}. \end{gather*}
\end{Proposition}

\subsection{Kleinian geometric structures}

The following proposition gives a tool that can be used to construct many interesting manifolds carrying geometric structures.

\begin{Proposition} \label{prop:geometry of coverings}
Let $M$ be an $(X,\sG)$-manifold, and $\Gamma < \Aut_{(X,\sG)}(M)$ be a subgroup acting properly discontinuously and freely on $M$. Then $M/\Gamma$ is a manifold, and there exists a unique $(X,\sG)$-structure on $M/\Gamma$ such that the quotient $M \ra M/\Gamma$ is an $(X,\sG)$-map.

Conversely, let $M$ be an $(X,\sG)$-manifold, and let $\pi\colon \bar{M} \ra M$ be a covering map. Then there exists a unique $(X,\sG)$-structure on $\bar{M}$ such that $\pi$ is an $(X,\sG)$-map.
\end{Proposition}

\begin{Definition}Let $\Gamma < \sG$ be a discrete subgroup. A~\defin{domain of discontinuity} for $\Gamma$ is an open subset $\Omega \subset X$ that is $\Gamma$-invariant and such that $\Gamma$ acts properly discontinuously on $\Omega$.
\end{Definition}

By applying Proposition~\ref{prop:geometry of coverings}, if $\Omega$ is a domain of discontinuity for $\Gamma$ and $\Gamma$ acts freely on~$\Omega$ (which is always true if $\Gamma$ is torsion-free), then the quotient $\Omega / \Gamma$ is a manifold with an $(X,\sG)$-structure.

\begin{Definition}The geometric structures of the form $\Omega/\Gamma$ described above are called \defin{Kleinian $(X,\sG)$-structures}.
\end{Definition}

The theory of Anosov representations, introduced by Labourie \cite{AnosovFlowsLabourie} and Guichard--Wien\-hard~\cite{GWDomainsofDiscont}, gives methods for constructing interesting Kleinian geometric structures. We will not give here the complete definition of Anosov representations, we will only recall some of their properties. Let $\sG$ be a semi-simple Lie group and $\Gamma$ be a Gromov-hyperbolic group. Anosov representations $\rho\colon \Gamma \ra \sG$ are defined with reference to a parabolic subgroup $\sP \subset \sG$, they will be called $\sP$-Anosov representations. One property of a~$\sP$-Anosov representation $\rho$ is the existence of a $\rho$-equivariant map
\begin{gather*}\xi\colon \ \partial_\infty \Gamma \ra \sG/\sP\end{gather*}
which must, by definition, satisfy some special properties. Here with $\partial_\infty \Gamma$ we denote the boundary at infinity of $\Gamma$, defined by Gromov \cite{Gromov} for hyperbolic groups. The $\sP$-Anosov representations form an open subset of the character variety (see Section~\ref{sec:character varieties}):
\begin{gather*}\sP\text{-}\Anosov(\pi_1(S),\sG) \subset \Xx(\pi_1(S),\sG).\end{gather*}

When $(X,\sG)$ is a geometry of parabolic type, whose isotropy group might be different from~$\sP$, there is a very rich theory giving sufficient conditions for a~$\sP$-Anosov representation to admit a~domain of discontinuity $\Omega \subset X$, which is, in the best cases, co-compact. The domain $\Omega$ is defined using the map $\xi$. This theory was founded by Guichard--Wienhard~\cite{GWDomainsofDiscont}, and was improved and extended by Kapovich--Leeb--Porti \cite{KLPAnosov1}. For an example of how this works, see Section~\ref{subsec:dod}. In this way, it is possible to construct many examples of Kleinian geometric structures on closed manifolds for geometries of parabolic type.

One limitation of this method is that even if we construct an $(X,\sG)$-manifold $M = \Omega/\rho(\Gamma)$, we have no idea what the topology of $M$ is. Other techniques are needed to get a good understanding of these geometric manifolds, see for example Theorem~\ref{thm:dod}.

\subsection{Developing maps and holonomies}

The Kleinian geometric structures are the easiest to understand, but not all geometric structures are Kleinian. To work with general geometric structures, we introduce here the tools of developing maps and holonomy representations.

\begin{Lemma}
Let $N$ be a simply-connected manifold $(X,\sG)$-manifold. Then $N$ admits a global $(X,\sG)$-map
\begin{gather*}D\colon \ N \ra X\end{gather*}
unique up to post-composition by an element of $\sG$.
\end{Lemma}
\begin{proof}
Choose a point $n_0\in N$, and a chart $(U_0,\varphi_0)$ around $N$. We will extend $\varphi_0\colon U_0\ra X$ to a map $D$ defined on $N$ such that $D|_{U_0} = \varphi$. For every point $n\in N$, we will define the value $D(n)\in X$ in the following way. Choose a path $\gamma\colon [0,1]\ra N$ such that $\gamma(0)=n_0$, $\gamma(1)=n$. We can find charts $(U_1,\varphi_1), \dots, (U_k,\varphi_k)$ and points $t_0, \dots, t_k, s_0, \dots, s_k \in [0,1]$ such that
\begin{enumerate}\itemsep=0pt
\item[1)] $0=t_0 < t_1 < s_0 < t_2 < s_1 < t_3 < \dots < s_{k-2} < t_k < s_{k-1} < s_k = 1$,
\item[2)] $\gamma([t_0,s_0)) = U_0\cap \gamma([0,1])$,
\item[3)] $\gamma((t_i,s_i)) = U_i \cap \gamma([0,1])$,
\item[4)] $\gamma((t_k,s_k]) = U_k \cap \gamma([0,1])$.
\end{enumerate}
The path $\gamma((t_i,s_{i-1}))$ is contained in a connected component $C_i$ of $U_{i-1} \cap U_i$. There exists a~$g_i\in \sG$ such that $\varphi_{i-1}\circ \varphi_i^{-1}|_{\varphi_i(C_i)} = g_i|_{\varphi_i(C_i)}$.
We define
\begin{gather*}D(n) = g_1\circ g_2\circ \dots \circ g_k \circ \varphi_{k}(n).\end{gather*}
Now it is necessary to show that the value $D(n)$ is independent on the choice of the charts $(U_i,\varphi_i)$, with $i\geq 1$. This is an application of the principle of unique analytic continuation. Then, we need to show that $D(n)$ does not depend on the choice of the curve $\gamma$. This comes from the fact that~$M$ is simply-connected, hence every other curve $\gamma'$ is homotopic to~$\gamma$ relatively to the end-points. This defines a global $(X,\sG)$-map $D$, which depends only on the choice of~$(U_0,\varphi_0)$. The principle of unique analytic continuation gives the uniqueness of~$D$ up to an element of~$\sG$.
\end{proof}

Given an $(X,\sG)$-manifold $M$, we denote its universal covering by $\widetilde{M}$. By Proposition \ref{prop:geometry of coverings}, $\widetilde{M}$~inherits an $(X,\sG)$-structure from $M$. Since $\widetilde{M}$ is simply-connected, there is a global $(X,\sG)$-map
\begin{gather*}D\colon \ \widetilde{M} \ra X \end{gather*}
unique up to post-composition by an element of $\sG$.
\begin{Definition}The map $D$ is called the \defin{developing map} of the $(X,\sG)$-manifold $M$.
\end{Definition}

The fundamental group $\pi_1(M)$ acts on $\widetilde{M}$ by deck transformations. This action preserves the $(X,\sG)$-structure on $\widetilde{M}$, hence we have the inclusion $\pi_1(M)<\Aut_{(X,\sG)}(\widetilde{M})$. The composition of an element $\gamma\in\pi_1(M)$ with the developing map~$D$ is again a developing map, hence there exists an element $h(\gamma)\in \sG$ such that
\begin{gather*} D \circ \gamma = h(\gamma) \circ D. \end{gather*}
The map $h\colon \pi_1(M)\ra \sG$ is a group homomorphism, and the formula above tells us that the developing map is $h$-equivariant.

\begin{Definition}The group homomorphism $h$ is called the \defin{holonomy representation} of the $(X,\sG)$-manifold $M$. The pair $(D,h)$ is called the \defin{developing pair} of the $(X,\sG)$-manifold~$M$.
\end{Definition}

\begin{Example}In the case of a Kleinian geometric structure $M=\Omega/\Gamma$, for some $\Omega \subset X$ and $\Gamma < \sG$, the developing map is a covering
\begin{gather*}D\colon \ \widetilde{M} \ra \Omega\end{gather*}
and the holonomy representation is a homomorphism
\begin{gather*}h\colon \ \pi_1(M)\ra \Gamma \end{gather*}
such that $\ker(h)=\pi_1(\Omega)$.
\end{Example}

If we change the developing map by post-composing it with an element $g\in\sG$, the holonomy representation changes by conjugation by $g$. In other words, the group $\sG$ acts on the developing pairs in the following way:
\begin{gather*}g \cdot (D,h) = \big(g\circ D, g h g^{-1}\big). \end{gather*}
The developing pair $(D,h)$ of the $(X,\sG)$-manifold $M$ is well defined up to this action of $\sG$. The developing pair completely determines the $(X,\sG)$-structure on $M$, as we will now see.

\begin{Definition}Let $M$ be a manifold without a specified $(X,\sG)$-structure. We will say that a pair $(D,h)$ is an $(X,\sG)$-\emph{developing pair} for $M$ if
\begin{enumerate}\itemsep=0pt
\item[1)] $h$ is a representation $h\colon \pi_1(M)\ra \sG$,
\item[2)] $D$ is an $h$-equivariant local diffeomorphism.
\end{enumerate}
\end{Definition}

Given an $(X,\sG)$-developing pair $(D,h)$ for $M$, we can construct an $(X,\sG)$-structure in the following way: let $U$ be a simply-connected open subset of $M$, and let $s\colon U\ra \widetilde{M}$ be a section of the universal covering. Assume that $U$ is small enough, so that $s(U)$ is an open subset where~$D$ is a diffeomorphism. Then $(U, D\circ s)$ is a chart, and the collection of all the charts of this type forms an atlas for a $(X,\sG)$-structure on $M$. This is the unique $(X,\sG)$-structure on $M$ with developing pair~$(D,h)$.

\subsection{Parameter spaces}

Given a fixed manifold $M$, we want to define a parameter space of all $(X,\sG)$-structures on~$M$.

\begin{Definition}We will say that two $(X,\sG)$-structures on $M$ are \defin{isotopic} if there exists a~diffeomorphism $f\colon M\ra M$ isotopic to the identity, which is an $(X,\sG)$-isomorphism between the first structure and the second.
\end{Definition}

We will denote by $\Dd_{(X,\sG)}(M)$ the set of all the $(X,\sG)$-structures on $M$ up to isotopy. The topology on $\Dd_{(X,\sG)}(M)$ is given by the $C^\infty$-topology on the corresponding developing maps. Let's see this in more detail.

Consider the space $\Dev_{(X,\sG)}(M)$ of all $(X,\sG)$-developing pairs $(D,h)$ for $M$. This space is endowed with the $C^\infty$-topology on the developing maps. Given a sequence of developing pairs $(D_k,h_k)$, it is easy to check that if the sequence $(D_k)$ converges to $D_0$ in the $C^\infty$-topology, then the sequence $(h_k)$ converges point-wise to~$h_0$.

Choose a point $m\in M$, and consider the group $\Diffeo_0(M,m)$ of all diffeomorphisms of~$M$ that fix the point $m$ and are isotopic to the identity. Every element of this group can be lifted in a unique way to a diffeomorphism of $\widetilde{M}$ that fixes the fiber over $m$. In this way, $\Diffeo_0(M,m)$ acts on $\widetilde{M}$. The group $\Diffeo_0(M,m) \times G$ acts on $\Dev_{(X,\sG)}(M)$ in the following way:
\begin{gather*}(f,g)\cdot (D,h) = \big(g\circ D \circ f, g h g^{-1}\big). \end{gather*}
We have that
\begin{gather*}\Dd_{(X,\sG)}(M) = \Dev_{(X,\sG)}(M)/ \Diffeo_0(M,m) \times G.\end{gather*}
In this way, the parameter space of $(X,\sG)$-structures on $M$ inherits the quotient topology.

\section{Representations and flat bundles} \label{sec:character varieties}

In this section, we review the correspondence between conjugacy classes of representations of the fundamental group of a manifold and isomorphism classes of flat bundles. We tried to keep the required Lie theory to a minimum, anyway, for all the Lie-theoretical notions, the reader can refer to \cite{Helgason,HumphreysBook,SerreBook}.

\subsection{Character varieties}

Let $\Gamma$ be a finitely generated group. Here, the most interesting case is when $\Gamma$ is the fundamental group of a closed manifold, but for the moment it can be arbitrary. Let $\sG$ be a reductive Lie group with Lie algebra $\fg$. We will denote by $\Hom(\Gamma,\sG)$ the set of all representations (i.e., group homomorphisms) of $\Gamma$ in $\sG$, endowed with the topology of point-wise convergence of representations.

\begin{Definition}A \defin{reductive representation} of $\Gamma$ in $\sG$ is a representation $\rho\colon \Gamma \ra \sG$ such that the induced action on $\fg$ given by the adjoint representation is completely reducible.
\end{Definition}

\begin{Example}If $\sG$ is a linear group, then $\rho$ is reductive if and only if it is completely reducible.
\end{Example}

We will denote by $\Hom^*(\Gamma,\sG)$ the subspace of all reductive representations of $\Gamma$ in $\sG$. The group $\sG$ acts on $\Hom^*(\Gamma,\sG)$ by conjugation, and the action is proper. We will denote the quotient by this action by
\begin{gather*}\Xx(\Gamma,\sG) = \Hom^*(\Gamma,\sG) / \sG. \end{gather*}

\begin{Definition}The space $\Xx(\Gamma,\sG)$ is called the \defin{character variety} of $\Gamma$ in $\sG$.
\end{Definition}

Character varieties are Hausdorff topological spaces. They are in general not manifolds since they can have singularities, but they are always locally contractible.

When $\Gamma = \pi_1(S)$, for a closed orientable surface $S$ of genus $g\geq 2$ and $\sG$ is a real Lie group, there are results describing the topology of some connected components of the character varieties.

\begin{Example} \label{exa:special representations}\quad
\begin{enumerate}\itemsep=0pt
\item When $\sG=\sPSL(2,\R)$, Goldman \cite{GoldmanThesis} used a topological invariant, the Euler number, to classify the connected components of the character variety: it has $4g-3$ connected components corresponding to the values of the Euler number from $2-2g$ to $2g-2$. Moreover Goldman proved that a representation in $\sPSL(2,\R)$ is discrete and faithful if and only if it has Euler number $\pm(2g-2)$. Such representations are called \defin{Fuchsian representations}, and they form two connected components of the character variety, each of whom is a copy of the Teichm\"uller space $\Tt(S)$ of the surface. Hitchin \cite{selfduality} described the topology of all the connected components with non-zero Euler number.
\item Similarly, when $\sG=\sPGL(2,\R)$, the set of discrete and faithful representations, again called Fuchsian representations, forms a connected component of the character variety which is a~copy of the Teichm\"uller space $\Tt(S)$ of the surface. This component is then homeomorphic to $\R^{6g-6}$, and it is also denoted by $\Hit(S,2)$, see below.
\item Consider now the case when $\sG=\sPGL(n,\R)$. A \defin{Fuchsian representation} in $\sPGL(n,\R)$ is defined as the composition of a Fuchsian representation in $\sPGL(2,\R)$ with the irreducible representation $\sPGL(2,\R)\ra \sPGL(n,\R)$. This construction gives an embedding of the Teichm\"uller space $\Tt(S)$ in $\Xx(\pi_1(S), \sPGL(n,\R))$, whose image is called the \defin{Fuchsian locus}. Hitchin \cite{liegroupsteichmuller} proved that the connected component of the character variety containing the Fuchsian locus is homeomorphic to $\R^{(n^2-1)(2g-2)}$. This component is called the \defin{Hitchin component}, and denoted by $\Hit(S,n)$. Labourie~\cite{AnosovFlowsLabourie} described the geometry of the representations in this component. The Hitchin components share many properties with the Teichm\"uller spaces, hence they are sometimes called \defin{higher Teichm\"uller spaces}, and they give the name to higher Teichm\"uller theory.
\item Hitchin~\cite{liegroupsteichmuller} defined special components in the character varieties of all split real simple Lie groups $\sG$. They are homeomorphic to $\R^{\dim(\sG)(2g-2)}$.
\item The Euler number can be generalized to representations into all Lie groups of Hermitian type, in this case it is called the Toledo number, see Toledo \cite{Toledo}. Representations with maximal value of the Toledo number are called \defin{maximal representations}, and they form a~union of connected components in the corresponding character varieties. This is another way to generalize Fuchsian representations to higher rank Lie groups. For $\sG = \sSp(4,\R)$ and $\sPSp(4,\R)$, an explicit description of the topology of the maximal components was determined in a joint work with Brian Collier \cite{SO23LabourieConj}.
\end{enumerate}
\end{Example}

When $\sG$ is a complex Lie group, we don't have explicit descriptions of connected components of character varieties. But there are at least some special open subsets that are particularly interesting. For example, in the character variety of $\Xx(\pi_1(S),\sPGL(2,\C))$ we have the open subset of \defin{quasi-Fuchsian representations}, denoted by~$\QFuch(S)$. Quasi-Fuchsian representations can be defined as those representations whose action on $\CP^1$ is topologically conjugate to the action of a Fuchsian representation on~$\CP^1$. The open subset $\QFuch(S)$ is homeomorphic to~$\R^{12g-12}$.

Hitchin components generalize Teichm\"uller spaces to higher rank Lie groups. In a similar way, there are some special open subsets of the character varieties of a simple complex Lie group $\sG$ that generalize the space of quasi-Fuchsian representations. We will call it the space of \defin{quasi-Hitchin representations}. To define them, consider the open subset
\begin{gather*}\sB\text{-}\Anosov(\pi_1(S),\sG) \subset \Xx(\pi_1(S),\sG)\end{gather*}
consisting of all $\sB$-Anosov representations, where $\sB$ is the Borel subgroup of $\sG$. The space of quasi-Hitchin representations is then defined as the connected component of $\sB\text{-}\Anosov(\pi_1(S),\sG)$ containing the Hitchin component of the split real form of $\sG$. For the group $\sPGL(n,\C)$, we will denote the space of quasi-Hitchin representations by $\QHit(S,n)$.

\subsection{Flat bundles}

Let $M$ be a manifold, and let $X$ be a manifold endowed with an effective action of $\sG$.

\begin{Definition}
A \defin{fiber bundle} on $M$ with \defin{structure group} $\sG$ and \defin{fiber} $X$ (also called a \defin{$\sG$-bundle} with fiber $X$) is a manifold $B$ with a smooth map $\pi\colon B \ra M$ and a maximal $\sG$-atlas for $\pi$. Recall that a \defin{$\sG$-atlas} is a set of charts $\{(U_i,\varphi_i)\}$ where the $U_i$s are open subsets of $M$ which cover $M$, and
\begin{gather*}\varphi_i\colon \ \pi^{-1}(U_i) \ra U_i \times X \end{gather*}
is a diffeomorphism that intertwines $\pi|_{U_i}$ and the projection on the first factor. The $\varphi_i$s must be \defin{$\sG$-compatible} in the following sense: the maps
\begin{gather*}\varphi_i \circ \varphi_j^{-1}\colon \ (U_i \cap U_j) \times X \ra (U_i \cap U_j) \times X \end{gather*}
are of the form $\varphi_i \circ \varphi_j^{-1}(m,x) = (m,t_{ij}(m)x)$, where $t_{ij}(m)\in \sG$. The functions
\begin{gather*}t_{ij}\colon \ U_i \cap U_j \ra \sG\end{gather*}
are called the \defin{transition functions} of the atlas.
\end{Definition}

It is interesting to remark that the bundle is determined up to isomorphism by the transition functions, and that the transition functions don't depend at all on the space $X$. This has the following consequence: if $X, Y$ are manifolds with effective actions of $\sG$, then a bundle~$B$ with structure group~$\sG$ and fiber~$X$, determines a~bundle~$B(Y)$ with the same structure group and fiber $Y$. The bundle $B(Y)$ is defined as the bundle with fiber $Y$ having the same transition functions as the bundle~$B$.

More generally, given a group homomorphism $q\colon \sG \ra \sH$, assume that $X$ is a manifold with an effective $\sG$-action, $Y$ is a manifold with an effective $\sH$-action, and $B$ is a $\sG$-bundle with fiber~$X$. We can apply the construction given above by composing the transition functions of $B$ with the homomorphism $q$. This produces an $\sH$-bundle $B(Y)$ with fiber $Y$.

\begin{Definition}The bundle $B(Y)$ is called the \defin{associated bundle} to $B$ with fiber $Y$.
\end{Definition}

\begin{Example}\quad
\begin{enumerate}\itemsep=0pt
\item The most important example is the special case when $X = \sG$, acting on itself on the left. A bundle with structure group and fiber $\sG$ is called a \defin{principal} $\sG$-bundle.
\item Another fundamental example is the case when the Lie group $\sG$ is a \defin{linear group}, i.e., when it can be embedded as a Lie subgroup of $\sGL(n,\R)$ or $\sGL(n,\C)$. In this case it admits an effective linear action on $V=\R^n$ or $\C^n$, and a $\sG$-bundle with fiber $V$ is called a \defin{vector bundle}. Starting from every bundle $B$ with structure group $\sG$ and some fiber, we can construct the associated vector bundle $B(V)$.
\item Similarly, a \defin{projective group} is a Lie group $\sG$ that can be embedded as a Lie subgroup of $\sPGL(n+1,\R)$ or $\sPGL(n+1,\C)$. In this case it admits an effective projective action on $P = \RP^{n}$ or $\CP^{n}$, and a $\sG$-bundle with fiber $P$ is called a \defin{projective bundle}. Starting from every bundle $B$ with structure group $\sG$ and some fiber, we can construct the associated projective bundle $B(P)$.
\end{enumerate}
\end{Example}

There is a close relationship between vector bundles and projective bundles. Assume that $\sG$ is a linear group, $X = \R^{n+1}$ or $\C^{n+1}$, $\sH$ is the corresponding projectivized group and $Y=\RP^{n}$ or $\CP^{n}$. From every vector bundle $E$ with structure group $\sG$, we can construct the associated projective bundle $E(Y)$. We will denote $E(Y)$ by $\bP(E)$, the \defin{projectivized bundle} of $E$.

\begin{Definition}A \defin{flat structure} on a $\sG$-bundle $B$ is an atlas of $B$ satisfying the additional condition that all transition functions are locally constant, and maximal among all atlases sa\-tisfying this additional condition. A bundle with a flat structure is called a \defin{flat bundle}, or a \defin{local system}.
\end{Definition}

The flat structure is just a special atlas, hence, as explained above, it does not depend on the fiber. It is thus possible to construct \defin{associated bundles} and to replace a flat bundle $B$ with fiber $X$ by a flat bundle $B(Y)$ with fiber $Y$.

If $\sG$ is a linear group acting on a vector space $V=\R^n$ or $\C^n$, starting from every flat bundle~$B$ with fiber $X$, we can change fiber and construct the vector bundle $B(V)$, with a flat structure. A~flat structure on a vector bundle can be described by a \defin{flat connection}, i.e., a connection with vanishing curvature form. This description is useful for doing computations.

A flat bundle $B$ with fiber $X$ has a well defined foliation, called the \defin{parallel foliation}, that can be described in local charts: in $\pi^{-1}(U_i)$, for every $x \in X$ there is a local leaf given by $\varphi_i^{-1}(U_i \times \{x\})$. The local foliations defined by the different charts all match up, giving rise to a~global foliation.

A \defin{local parallel section} of a flat bundle is a section $s\colon U \ra B$ defined on an open subset $U$, which is locally constant when restricted to every chart of the flat structure. In other words, it is a section whose image is contained in a leaf of the parallel foliation. Similarly, given a curve $\gamma\colon [0,1]\ra M$, a~\defin{parallel section along} $\gamma$ is a section along $\gamma$ which is locally constant in the charts.

Using parallel sections, we can define the \defin{parallel transport operator} along a curve $\gamma\colon [0,1]\ra M$: it is an operator
\begin{gather*}P_\gamma\colon \ \pi^{-1}(\gamma(0)) \ra \pi^{-1}(\gamma(1)) \end{gather*}
defined in the following way: given $x_0\in \pi^{-1}(\gamma(0))$, there exists a unique parallel section~$s$ along~$\gamma$ such that $s(0)=x_0$. We define $P_\gamma(x_0) = s(1)$. The parallel transport $P_\gamma$ only depends on the homotopy class of~$\gamma$ relative to the end points.

\subsection{Monodromy}

Let $\pi\colon B\ra M$ be a flat $\sG$-bundle with fiber $X$. Given a base point $m_0 \in M$, we can use a chart to identify the fiber $\pi^{-1}(m_0)$ with $X$. If we change chart, this identification changes by the action of an element of $\sG$. Now the parallel transport $P_\gamma$ along a loop $\gamma$ based at~$m_0$ is a~map $P_\gamma\colon X \ra X$, which agrees with the action of an element of $\sG$, hence we can write $P_\gamma\in\sG$. Since~$P_\gamma$ only depends on the homotopy class of $\gamma$, we get a map
\begin{gather*}P\colon \ \pi_1(M,m_0) \ra \sG. \end{gather*}

This map behaves well under composition of loops, hence it is a representation. If we change the chart around $m_0$, the representation changes by conjugation by an element of~$\sG$.

\begin{Definition}The representation $P$ is called the \defin{monodromy representation} of the bundle. We will say that a flat bundle is \defin{reductive} if the monodromy representation is reductive.
\end{Definition}

We will denote by $\Flat(M,G,X)$ the space of all reductive flat $G$-bundles with fiber $X$ up to isomorphism. If $X$, $Y$ are manifolds with effective actions of $\sG$, the associated bundle construction gives a natural bijection $\Flat(M,G,X)\ra \Flat(M,G,Y) $. Hence, we can suppress the $X$ in the notation, and consider the space $\Flat(M,G)$, parametrizing reductive flat $\sG$-bundles with any fixed fiber $X$.

The monodromy representation gives a map
\begin{gather*}\Pp\colon \ \Flat(M,G) \ra \Xx(\pi(M), G). \end{gather*}

\begin{Proposition} \label{prop:monodromy}The map $\Pp$ is a bijection.
\end{Proposition}
\begin{proof}The inverse map is given by the following construction. Let $\rho\colon \pi_1(M)\ra \sG$ be a~re\-pre\-sentation. This gives an action of $\pi_1(M)$ on $\widetilde{M}\times X$, acting on the first factor by deck transformations and on the second factor via $\rho$. This action is properly discontinuous and free because the action on the first factor has these properties. Hence, we can construct the manifold
\begin{gather*}X_\rho = \big(\widetilde{M}\times X\big)/\pi_1(M). \end{gather*}
The projection on the first factor induces a map $p\colon X_\rho\ra M$ which turns $X_\rho$ into a $\sG$-fiber bundle with fiber $X$. Moreover, the product $ \widetilde{M}\times X$ induces a flat structure on the bundle. It is easy to show that the flat bundle $X_\rho$ has monodromy $\rho$, and that any other flat $\sG$-bundle with fiber~$X$ and monodromy $\rho$ is isomorphic to~$X_\rho$.
\end{proof}

\section{The graph of a geometric structure} \label{sec:graph}

In this section, we will see how geometric structures correspond to flat bundles with a transverse section. The flat bundle encodes the holonomy representation of the geometric structure, and the transverse section encodes the developing map. This construction is described in detail in Goldman's notes~\cite{GoldmanGSAVOR}.

\subsection{Sections and equivariant maps}

Let $\rho\colon \pi_1(M) \ra \sG$ be a representation, and consider the space $\Equiv(\rho,X)$ of smooth $\rho$-equivariant maps from the universal covering $\widetilde{M}$ to $X$, endowed with the $C^\infty$-topology. Let~$B$ be the a flat bundle over $M$ with fiber $X$ and holonomy $\rho$, and consider the space $\Gamma(M, B)$ of smooth sections of~$B$, endowed with the $C^\infty$-topology.

\begin{Proposition}There is a natural homeomorphism between $\Gamma(M, B)$ and $\Equiv(\rho,X)$.
\end{Proposition}
\begin{proof}Recall first that $B$ is isomorphic to $X_\rho$, the bundle defined in the proof of Proposi\-tion~\ref{prop:monodromy}. From the construction of $X_\rho$, we can see that the pull-back of $X_\rho$ to $\widetilde{M}$ is isomorphic to a~product $\widetilde{M}\times X$, with the product flat structure.

A section $s\in \Gamma(M, X_\rho)$ can be pulled back to a section $\widetilde{s}$ of $\widetilde{M}\times X$. A section of a product bundle is just a map $\widetilde{s}\colon \widetilde{M}\ra X$. The fact that $\widetilde{s}$ is a pull-back tells us that this map is $\rho$-equivariant. This gives a map between $\Gamma(M, X_\rho)$ and $\Equiv(\rho,X)$.

To find the inverse of this map, just notice that a $\rho$-equivariant map $f\colon \widetilde{M}\ra X$ is a section of the product bundle $\widetilde{M}\times X$. The fact that $f$ is $\rho$-equivariant implies that it passes to the quotient, giving a section $[f]$ of $X_\rho$.

To check that the maps are continuous, we can work locally on small open sets of $M$ which are well covered by the universal covering.
\end{proof}

\subsection{Transverse sections}

Let $\rho\colon \pi_1(M) \ra \sG$ be a representation, and $B$ be the flat bundle over $M$ with fiber $X$ and holonomy $\rho$.

\begin{Definition}A section $s\in \Gamma(M, B)$ is \defin{transverse} if it is transverse to the parallel foliation of the bundle.
\end{Definition}
\begin{Proposition}A section $s\in \Gamma(M, B)$ is transverse if and only if the corresponding $\rho$-equivariant map is
\begin{enumerate}\itemsep=0pt
\item[$1)$] an immersion if $\dim(M) \leq \dim(X)$,
\item[$2)$] a submersion if $\dim(M) \geq \dim(X)$.
\end{enumerate}
In particular, if $\dim(M)= \dim(X)$, then $s$ is transverse if and only if the corresponding $\rho$-equivariant map is a local diffeomorphism.
\end{Proposition}
\begin{proof}
Let $f\colon \widetilde{M}\ra X$ be the corresponding $\rho$-equivariant map, and let $\pi\colon \widetilde{M} \ra M$ denote the universal covering.
Let $v\in T_{x}\widetilde{M}$, and $v' = d\pi[v] \in T_{\pi(x)} M$. Then the differential of $f$ vanishes at $v$ if and only if the differential of $s$ at $v'$ is tangent to the parallel foliation.
\end{proof}

\begin{Definition}If $\dim(M) = \dim(X)$, a \defin{graph of an $(X,\sG)$-structure} is a pair $(B,s)$ where $B$ is a flat bundle over~$M$ with fiber $X$ and $s\in \Gamma(M, B)$ is a transverse section.
\end{Definition}

Graphs of $(X,\sG)$-structures correspond to $(X,\sG)$-developing pairs, which determine $(X,\sG)$-structures on~$M$.

Let's see this more explicitly in the case of real or complex projective structures. Let $\K$ be $\R$ or $\C$, and consider the geometry $\KP^{n}$. Given a representation $\rho\colon \pi_1(M)\ra \sPGL(n+1,\K)$, we want to construct a $\KP^n$-structure on $M$ with holonomy $\rho$. To do this, we need to consider the flat bundle $B$ over $M$ with fiber $\KP^{n}$ and holonomy $\rho$, and construct a transverse section of $B$.

This becomes more concrete when $\rho$ lifts to a representation $\bar{\rho}\colon \pi_1(M) \ra \sGL(n+1,\K)$. In this case, there is a flat vector bundle $E$ with holonomy $\bar{\rho}$ such that the projectivized bundle~$\bP(E)$ is isomorphic to~$B$. The flat structure on $E$ is described by a flat connection~$\nabla$. A~section of~$B$ is the same thing as a line subbundle of~$E$. The next proposition shows how it is possible to verify whether a section of~$B$ is transverse with a computation in local coordinates involving the derivatives with reference to the flat connection on~$E$.

\begin{Proposition} \label{prop:transversality condition} Let $E$ be a flat vector bundle of rank $n+1$ over $M$, and $L \subset E$ be a line subbundle. Then $L$ is a transverse section of $\bP(E)$ if and only if for every $m\in M$ there exists a coordinate neighborhood $U$ of $m$ $($with coordinates $x_1, \dots, x_k$, where $k=\dim(M))$ and a local non-vanishing section $s\colon U \ra L$ such that the local vector fields
\begin{gather*}s, \nabla_{\!\!\!\frac{\partial}{\partial x_1}} s, \dots, \nabla_{\!\!\!\frac{\partial}{\partial x_k}} s \end{gather*}
satisfy one of the following conditions:
\begin{enumerate}\itemsep=0pt
\item[$1)$] are linearly independent on $U$ if $\dim(M) \leq n$,
\item[$2)$] span every fiber over $U$ if $\dim(M) \geq n$.
\end{enumerate}
\end{Proposition}

\subsection{The holonomy map}

If $\sG$ is reductive, we can consider the subspace
\begin{gather*}\Dd^*_{(X,\sG)}(M) \subset \Dd_{(X,\sG)}(M) \end{gather*}
of all $(X,\sG)$-structures on $M$ with reductive holonomy. This subspace has a natural map to the character variety, given by the holonomy representation:
\begin{gather*}\Hol\colon \ \Dd^*_{(X,\sG)}(M) \ra \Xx(\pi_1(M),\sG).\end{gather*}

\begin{Theorem}[Thurston's holonomy principle]If $M$ is a closed manifold, the map $\Hol$ is open and it has discrete fiber.
\end{Theorem}
\begin{proof}See Goldman \cite{GoldmanGSAVOR}. The openness of the map $\Hol$ can be proved easily using graphs of $(X,\sG)$-structures.
\end{proof}

When $M$ is closed, the map $\Hol$ is very often a local homeomorphism, but not always (for a~counterexample, see Baues~\cite{BauesTorus}). This issue needs to be better understood:

\begin{Question}[refined Thurston's holonomy principle]\quad
\begin{enumerate}\itemsep=0pt
\item Is it true that the map $\Hol$ is always a branched local homeomorphism?
\item What are some sufficient conditions for it to be a local homeomorphism?
\end{enumerate}
\end{Question}

Other important questions are raised by the fact that the map $\Hol$ is in general neither injective nor surjective.
\begin{Question} \label{question:holonomy fibers}
Let $\rho\colon \pi_1(M)\ra \sG$ be a representation.
Is there an $(X,\sG)$-structure on $M$ with holonomy $\rho$? And in the affirmative case, how many are there?
\end{Question}
A complete answer to Question \ref{question:holonomy fibers} is known only in very special cases, for example for $\CP^1$-structures on closed surfaces (see Gallo--Kapovich--Marden \cite{GalloKapovichMarden}, Goldman \cite{GoldmanFuchsianHol}, Baba \cite{BabaGrafting1,BabaGrafting2}). In the case of $\RP^2$-structures on closed surfaces a partial answer is given in Choi--Goldman \cite{ChoiGoldmanClassification}. The character varieties are much easier to understand than the parameter spaces $\Dd^*_{(X,\sG)}(M)$, hence, if we can obtain a better understanding of Question~\ref{question:holonomy fibers}, we can use our knowledge about representations to understand parameter spaces of geometric structures.

A plan to answer these questions can be the following: given a representation $\rho$, we construct the corresponding flat bundle, and we then try to understand all the possible transverse sections. An obstacle is that even for representations that we know very well, we don't always understand the corresponding flat bundle well enough to see the transverse sections. This is the point when Higgs bundles can be very useful: they can give an explicit description of the flat bundle.

\section{How to use Higgs bundles?} \label{sec:higgs bundles}

We will show in some simple examples how Higgs bundles can be used to construct geometric structures on manifolds. The flat connection can be expressed in terms of solutions of Hitchin's equations and the transverse section can be constructed from the study of the holomorphic structure of the vector bundle. This idea first appeared in Baraglia's Ph.D.~Thesis~\cite{BaragliaThesis}.

\subsection[$\sSL(2,\R)$-Higgs bundles]{$\boldsymbol{\sSL(2,\R)}$-Higgs bundles}

In this subsection we will describe all the $\sSL(2,\R)$-Higgs bundles. We will use this description in Sections~\ref{subsec:hyperbolic structures}, \ref{subsection:almost fuchsian}, \ref{subsec:ads}, and \ref{subsec:projective structures Higgs}. Let $\Sigma$ be a closed Riemann surface.

\begin{Definition}An \defin{$\sSL(2,\R)$-Higgs bundle} on $\Sigma$ is a tuple $(E,Q,\omega,\varphi)$, where
\begin{itemize}\itemsep=0pt
\item[1)] $E$ is a holomorphic vector bundle on $\Sigma$ of rank $2$,
\item[2)] $Q\colon E\ra E^*$ is a holomorphic symmetric $\C$-bilinear form,
\item[3)] $\omega\in H^0\big(\Sigma,\Lambda^2 E\big)$ is a holomorphic $\C$-volume form such that $Q$ has volume $1$,
\item[4)] $\varphi\in H^0(\Sigma,\End(E)\otimes K)$ is $Q$-symmetric and satisfies $\tr(\varphi)=0$ (the \defin{Higgs field}).
\end{itemize}
\end{Definition}

The first three conditions say that $(E,Q,\omega)$ is a rank $2$ vector bundle with an $\sSO (2,\C)$-structure. In particular, $\Lambda^2 E = \Oo$.

The structure of such an $\sSL(2,\R)$-Higgs bundle can be made more explicit. This description was done by Hitchin \cite{selfduality}, who started from a different definition of $\sSL(2,\R)$-Higgs bundles. Consider the set of $Q$-isotropic vectors:
\begin{gather*}\Iso(Q) = \{v \in E \,|\, Q(v,v) = 0\}. \end{gather*}
In every fiber, this set is the union of two lines. $E$ has two line subbundles whose total spaces are given by:
\begin{gather*}L_+ = \left\{v \in \Iso(Q) \,|\, \forall\, w \in \Iso(Q){\setminus} \Span(v), \ i \frac{\omega(v,w)}{Q(v,w)} > 0 \right\}, \\
L_- = \left\{v \in \Iso(Q) \,|\, \forall\, w \in \Iso(Q){\setminus} \Span(v), \ i \frac{\omega(v,w)}{Q(v,w)} < 0 \right\}. \end{gather*}

Hence, we have $E=L_+ \oplus L_-$. The condition $\Lambda^2 E = \Oo$ now says that $L_+ = L_-^{-1}$. To simplify the notation, we write $L=L_+$, $L^{-1} = L_-$. The Higgs bundle can be written as
\begin{gather*}E = L \oplus L^{-1}, \qquad Q = \begin{pmatrix}
0 & 1\\
1 & 0
\end{pmatrix}, \qquad \omega = \frac{i}{\sqrt{2}}\begin{pmatrix}
0 & 1\\
-1 & 0
\end{pmatrix}, \qquad \varphi = \begin{pmatrix}
0 & a\\
b & 0
\end{pmatrix},\end{gather*}
where $a \in H^0\big(\Sigma, L^2 K\big)$, $b \in H^0\big(\Sigma, L^{-2} K\big)$. The condition for the Higgs bundle to be poly-stable is that:
\begin{enumerate}\itemsep=0pt
\item If $\deg(L) > 0$, then $b \neq 0$.
\item If $\deg(L) < 0$, then $a \neq 0$.
\item If $\deg(L) = 0$, then $a,b\neq 0$ or $a=b=0$.
\end{enumerate}
In the case when $a=b=0$, the Higgs bundle is strictly poly-stable, in all other cases it is stable. These conditions impose a restriction to the degree of $L$ for a poly-stable $\sSL(2,\R)$-Higgs bundle: $|\deg(L)|\leq g-1$ (Milnor--Wood inequality).

A poly-stable Higgs bundle where $L$ has maximal possible degree ($\deg(L)= g-1$) is called a~\defin{Fuchsian Higgs bundle}, and they correspond to Fuchsian representations. The stability condition $b\neq 0$ forces $L$ to be a square root of $K$ (we will write $L = K^{\frac{1}{2}}$). The section $b$ is a constant, and, up to gauge transformations we can assume $b=1$. The section $a$ is a quadratic differential, we will write $a=q_2 \in H^0\big(\Sigma,K^2\big)$.

Let $H$ be the Hermitian metric on $E$ that solves Hitchin's equations.
\begin{Proposition}[{\cite[Theorem~3.1]{AdSpaper}}] If an $\sSL(2,\R)$-Higgs bundle is stable, then
\begin{gather*}H = \begin{pmatrix}
h & 0\\
0 & h^{-1}
\end{pmatrix},\end{gather*}
for some real positive $h\in \Gamma(\Sigma, \bar{L}\otimes L)$.
\end{Proposition}
\begin{proof}Consider the Higgs field $Q^{-1}\varphi^T Q \in H^0(\Sigma,\End(E)\otimes K)$. Then the metric $\bar{Q}^T\big(H^T\big)^{-1}Q$ is a solution of Hitchin's equations for the Higgs bundle $\big(E,Q^{-1}\varphi^T Q\big)$. The fact that $\varphi$ is $Q$-symmetric means that $\varphi = Q^{-1}\varphi^T Q$, hence $H = \bar{Q}^T\big(H^T\big)^{-1}Q$. This, plus the condition $\det(H)=1$ implies the statement.
\end{proof}

Let $\ell$ be a local holomorphic frame for $L$. Denote by $\ell'$ the dual holomorphic frame on $L^{-1}$. The pair $(\ell,\ell')$ is a local frame for $E$. In this local frame, we can write the flat connection given by the solutions of Hitchin's equations in the following way:
\begin{gather*}\nabla = d + H^{-1}\partial H + \varphi + H^{-1}\bar{\varphi}^T H = d + \begin{pmatrix}
-\partial \log h & a + h^2 \bar{b}\\
b + h^{-2} \bar{a} & \partial \log h
\end{pmatrix}. \end{gather*}

The real structure is given by
\begin{gather*}\tau\colon \ E \ni \begin{pmatrix}
v_1\\
v_2
\end{pmatrix} \lra
\begin{pmatrix}
0 & h\\
h^{-1} & 0
\end{pmatrix}
\begin{pmatrix}
\bar{v_1}\\
\bar{v_2}
\end{pmatrix} =
\begin{pmatrix}
h\bar{v_2}\\
h^{-1}\bar{v_1}
\end{pmatrix} \in E.
 \end{gather*}

And the real locus is given by
\begin{gather*}E_\R = \{v\in E \,|\, \tau(v)= v\}.\end{gather*}

\subsection{Hyperbolic structures on surfaces} \label{subsec:hyperbolic structures}

Now we show the simplest example of how to use Higgs bundles to construct geometric structures with given holonomy. We start with a Fuchsian representation $\rho\colon \pi_1(S) \ra \sPSL(2,\R)$, and we want to construct an $\bH^2$-structure with holonomy $\rho$. We will first construct a $\CP^1$-structure with holonomy $\rho$, and we will then verify that this $\CP^1$-structure is actually an $\bH^2$-structure.

We choose a complex structure $\Sigma$ on $S$ and we consider the $\sSL(2,\R)$-Higgs bundle $(E,\varphi)$ corresponding to a lift of $\rho$ to $\sSL(2,\R)$. Since $\rho$ is Fuchsian, we know that
\begin{gather*}E = K^{\frac{1}{2}} \oplus K^{-\frac{1}{2}}, \qquad \varphi = \begin{pmatrix}
0 & q_2\\
1 & 0
\end{pmatrix}, \qquad q_2 \in H^0\big(\Sigma,K^2\big). \end{gather*}

To construct a $\CP^1$-structure on $\Sigma$, we need to choose a line subbundle, and prove that it gives a transverse section of the projectivized bundle $\bP(E)$. We can choose $K^{\frac{1}{2}}$ as a subbundle. We will use the transversality condition from Proposition \ref{prop:transversality condition}. Given a local section $s$ of $K^{\frac{1}{2}}$, we can compute the derivatives:
\begin{gather*}s = \begin{pmatrix}
1\\0
\end{pmatrix}, \ \ \ \ \ \nabla_{\!\!\!\frac{\partial}{\partial z}}s =
\begin{pmatrix}
-\partial\log h\\1
\end{pmatrix}, \ \ \ \ \ \nabla_{\!\!\!\frac{\partial}{\partial \bar{z}}}s =
\begin{pmatrix}
0\\h^{-2}\bar{q_2}
\end{pmatrix}.\end{gather*}

Here we computed the derivatives in the complex directions $\frac{\partial}{\partial z}$ and $\frac{\partial}{\partial \bar{z}}$, but to apply Proposition \ref{prop:transversality condition} we need to transform into derivatives in the real directions. This gives the following modified condition: the section $K^{\frac{1}{2}}$ is transverse if and only if
\begin{gather*}\forall\, A,B\in \C, \qquad A \nabla_{\!\!\!\frac{\partial}{\partial z}}s + \bar{A} \nabla_{\!\!\!\frac{\partial}{\partial \bar{z}}}s + B s = 0 \qquad \Rightarrow \qquad A=B=0. \end{gather*}
Substituting, we see that the section $K^{\frac{1}{2}}$ is transverse if and only if
\begin{gather*}\forall\, A,B\in \C, \qquad
\begin{cases}
-A \partial\log h + B &= 0,\\
A +\bar{A}h^{-2}\bar{q_2} &= 0
\end{cases}
\qquad \Rightarrow \qquad A=B=0.
\end{gather*}

If $A\neq 0$, the second equation is equivalent to
\begin{gather*}\frac{A}{\bar{A}} = -h^{-2}\bar{q_2}. \end{gather*}
This cannot be satisfied because of the following lemma:
\begin{Lemma}[Hitchin \cite{selfduality}] In the above setup, we have
\begin{gather*} \big| h^{-2}\bar{q_2} \big| < 1. \end{gather*}
\end{Lemma}
\begin{proof}
If $q_2=0$, this is obvious. Otherwise, it was proven by Hitchin \cite{selfduality} applying the maximum principle.
\end{proof}

We have found a graph of a $\CP^1$-structure $\big(\bP(E),K^{\frac{1}{2}}\big)$. We denote by $D\colon \widetilde{\Sigma} \ra \CP^1$ the corresponding developing map. We can now check that the image of this map never meets $\RP^1$, this is because we wrote the real structure $\tau$ explicitly, and it is easy to check that $K^{\frac{1}{2}}$ is never in the real locus:
\begin{gather*}\tau\begin{pmatrix}
1\\0
\end{pmatrix} = \begin{pmatrix}
0\\h^{-1}
\end{pmatrix}. \end{gather*}
Hence we have a developing map
\begin{gather*}D\colon \ \widetilde{\Sigma} \ra \bH^2.\end{gather*}
The holonomy is in $\sPSL(2,\R)$ and hence this $\CP^1$-structure is actually an $\bH^2$-structure.

The map $D$ actually coincides with the harmonic map to the symmetric space coming from solving Hitchin's equations. The fact that $D$ is a local diffeo was proved by Sampson~\cite{Sampson}, Wolf~\cite{TeichOfHarmonic} and Hitchin~\cite{selfduality}. The proof given here is Hitchin's proof.

The case when $q_2=0$ is the easiest, but it is particularly interesting. Fuchsian Higgs bundles with $q_2=0$ are called \defin{uniformizing Higgs bundles}, because they give an alternative proof of a~version of the uniformization theorem. This was done by Hitchin~\cite{selfduality} with essentially the same proof we give here, but without mentioning geometric structures.

\begin{Theorem}[uniformization theorem] Every complex structure $\Sigma$ on a closed surface~$S$ admits a conformal Riemannian metric of constant curvature~$-1$.
\end{Theorem}
\begin{proof}Choose a square root $K^{\frac{1}{2}}$ of the canonical bundle, and take the uniformizing Higgs bundle with that square root. The equivariant map $D$ constructed above is now conformal: to see this, notice that
\begin{gather*} \nabla_{\!\!\!\frac{\partial}{\partial \bar{z}}}s = \begin{pmatrix}
0\\h^{-2}\bar{q_2}
\end{pmatrix} = 0.\end{gather*}
Hence, the pull-back of the hyperbolic metric on $\bH^2$ is conformal, and it has curvature $-1$.
\end{proof}

\subsection{Almost-Fuchsian representations} \label{subsection:almost fuchsian}

Given a Fuchsian representation in the character variety $\Xx(\pi_1(S),\sPGL(2,\R))$, we want to deform it in $\QFuch(S) \subset \Xx(\pi_1(S),\sPGL(2,\C))$, the space of quasi-Fuchsian representations. These representations have a very interesting geometry, and they are holonomies of some very special $\CP^1$-structures called the quasi-Fuchsian $\CP^1$-structures.

\begin{Definition}
Consider a homeomorphism $f\colon \CP^1 \ra \CP^1$ that topologically conjugates the action of a Fuchsian representation with the action of a quasi-Fuchsian representation $\rho$. Then the open subset $f\big(\bH^2\big)$ is a domain of discontinuity for $\rho$, and $S = f\big(\bH^2\big)/\rho(\pi_1(S))$ is a surface with a $\CP^1$-structure which is called a \defin{quasi-Fuchsian $\CP^1$-structure}.
\end{Definition}

We would like to see the quasi-Fuchsian $\CP^1$-structures in terms of Higgs bundles, but we are not able to do this in full generality. We can see this for a special open subset of the quasi-Fuchsian representations, which is called the space of almost-Fuchsian representations. The material in this section is part of a joint work with Qiongling Li \cite{NilpotentCone}.

Let's start with a uniformizing Higgs bundle
\begin{gather*}\left(K^{\frac{1}{2}} \oplus K^{-\frac{1}{2}}, \ \begin{pmatrix}
0 & 0\\
1 & 0
\end{pmatrix} \right).\end{gather*}
We can deform this Higgs bundle for $\sSL(2,\R)$ to a Higgs bundle for $\sSL(2,\C)$ by changing the holomorphic structure of the vector bundle. Consider the vector bundle
\begin{gather*} E = K^{\frac{1}{2}} \oplus K^{-\frac{1}{2}} \end{gather*}
endowed with the following holomorphic structure:
\begin{gather*}\bar{\partial}_E = \bar{\partial} + \begin{pmatrix}
0 & 0\\
\beta & 0
\end{pmatrix},\end{gather*}
with $\beta \in \Omega^{0,1}\big(\Sigma,K^{-1}\big)$. In the formula, $\bar{\partial}$ is the standard holomorphic structure of the direct sum, which is modified by adding a correction term. Such a bundle is an extension
\begin{gather*}0 \ra K^{-\frac{1}{2}}\ra E \ra K^{\frac{1}{2}} \ra 0.\end{gather*}
These extensions are classified by the Dolbeault cohomology class $[\beta] \in H^1\big(\Sigma,K^{-1}\big)$ a space isomorphic, by Serre's duality, to the dual of the space of quadratic differentials on $\Sigma$. Diffe\-rent~$\beta$s in the same cohomology class give rise to isomorphic vector bundles. The choice of the representative $\beta$ in the class corresponds to a choice of a non-holomorphic section $K^{\frac{1}{2}} \ra E$, whose image is the non-holomorphic subbundle appearing in the direct sum.

We consider now the Higgs bundle $(E,\varphi)$, where
\begin{gather*}E = \big(K^{\frac{1}{2}} \oplus K^{-\frac{1}{2}}, \bar{\partial}_E \big), \qquad \varphi = \begin{pmatrix}
0 & 0\\
1 & 0
\end{pmatrix}. \end{gather*}
The Higgs bundles of this form are parametrized by the pair $(\Sigma,[\beta])$. For every quasi-Fuchsian representation $\rho$, there exists a lift $\bar{\rho}\colon \pi_1(S)\ra \sSL(2,\C)$ and a pair $(\Sigma,[\beta])$ such that the flat connection of the corresponding Higgs bundle has monodromy~$\bar{\rho}$ (see~\cite{SandersThesis}). This pair is not unique in general. Moreover, not all the pairs $(\Sigma,[\beta])$ give rise to a quasi-Fuchsian monodromy. It is an open problem to distinguish them:
\begin{Question}
Given a complex structure $\Sigma$, how can we characterize the classes $[\beta]{\in} H^1\!\big(\Sigma{,}K^{-1}\!\big)$ such that the flat connection of the corresponding Higgs bundle has quasi-Fuchsian monodromy?
\end{Question}

Answering this question was our initial motivation for trying to construct the quasi-Fuchsian $\CP^1$-structures using Higgs bundles, but, as explained above, we still cannot construct all of them.

Let's fix now a pair $(\Sigma,[\beta])$. To construct a $\CP^1$-structure on $\Sigma$, we need to choose a line subbundle. We choose the holomorphic subbundle $K^{-\frac{1}{2}}$, and we then have to verify the transversality conditions.

Denote by $H$ the solutions of Hitchin's equations for the corresponding Higgs bundle. We can choose the representative $\beta$ in the Dolbeault cohomology class in a way such that the non-holomorphic subbundle $K^{\frac{1}{2}}$ is $H$-orthogonal to the holomorphic subbundle $K^{-\frac{1}{2}}$. With this choice, we can write~$H$ as
\begin{gather*}H = \begin{pmatrix}
h^{-1} & 0\\
0 & h
\end{pmatrix}. \end{gather*}
We can now write the flat connection:
\begin{gather*}\nabla = d + \begin{pmatrix}
-\partial \log h & h^2 \big(\bar{1}+\bar{\beta}\big)\\
1+\beta & \partial \log h
\end{pmatrix}.
\end{gather*}

Given a local section $s$ of $K^{-\frac{1}{2}}$, we can compute the derivatives:
\begin{gather*}s = \begin{pmatrix}
0\\1
\end{pmatrix}, \qquad \nabla_{\!\!\!\frac{\partial}{\partial z}}s =
\begin{pmatrix}
h^2\bar{\beta}\\ \partial\log h
\end{pmatrix}, \qquad \nabla_{\!\!\!\frac{\partial}{\partial \bar{z}}}s =
\begin{pmatrix}
h^2\\0
\end{pmatrix}.\end{gather*}

As in the previous subsection, the transversality condition from Proposition \ref{prop:transversality condition} is equivalent to the following condition: the section $K^{-\frac{1}{2}}$ is transverse if and only if
\begin{gather*}\forall\, A,B\in \C, \qquad A \nabla_{\!\!\!\frac{\partial}{\partial z}}s + \bar{A} \nabla_{\!\!\!\frac{\partial}{\partial \bar{z}}}s + B s = 0 \qquad \Rightarrow \qquad A=B=0. \end{gather*}
Substituting, we see that the section $K^{-\frac{1}{2}}$ is transverse if and only if

\begin{gather*}\forall\, A,B\in \C, \qquad
\begin{cases}
Ah^2\bar{\beta} + \bar{A}h^{2} = 0,\\
A \ \partial\log h + B = 0
\end{cases}
\qquad \Rightarrow \qquad A=B=0.
\end{gather*}

If $A\neq 0$, the first equation is equivalent to
\begin{gather*}\frac{\bar{A}}{A} = \bar{\beta}. \end{gather*}
If $|\beta|<1$, this cannot be satisfied, hence the section is transverse. The condition $|\beta|<1$ is well known, see Uhlenbeck \cite{Uhlenbeck}:

\begin{Definition}
A representation $\rho\colon \pi_1(S)\ra \sPGL(2,\C)$ is called \defin{almost-Fuchsian} if it is the projectivization of the monodromy of the flat connection of a Higgs bundle associated with a~pair $(\Sigma,[\beta])$, with $|\beta|<1$.
\end{Definition}

Almost-Fuchsian representations are a special type of quasi-Fuchsian representations having very good analytic properties. Summarizing, we find the following:

\begin{Theorem}[Alessandrini--Li \cite{NilpotentCone}]Let $\rho\colon \pi_1(S)\ra \sPGL(2,\C)$ be an almost-Fuchsian representation corresponding to the Higgs bundle $(E,\varphi)$ defined by the pair $(\Sigma,[\beta])$. Then the holomorphic line subbundle $K^{-\frac{1}{2}} \subset E$ induces a quasi-Fuchsian $\CP^1$-structure with holonomy~$\rho$.
\end{Theorem}

\subsection{Convex real projective structures}

\begin{Definition}An $\RP^2$-structure on a closed surface $S$ is said to be a \defin{convex $\RP^2$-structure} if the developing map
\begin{gather*}D\colon \ \widetilde{S}\ra \RP^2 \end{gather*}
is a diffeomorphism with an open convex subset of $\RP^2$.
\end{Definition}

Examples of convex $\RP^2$-structures were given in Example \ref{exa:geometric manifolds}, where we have seen that every $\bH^2$-structure on $S$ produces such an $\RP^2$-structure via the Klein model.

The subset of $\Dd_{\RP^2}(S)$ consisting of convex real projective structures will be denoted by $\Dd^{\mathrm{conv}}_{\RP^2}(S)$. The holonomy of these structures is always reductive, hence we have
\begin{gather*}\Hol\colon \ \Dd^{\mathrm{conv}}_{\RP^2}(S) \ra \Xx(\pi_1(S), \sPGL(3,\R)) ).\end{gather*}

Goldman \cite{GoldmanConvex} proved that $\Dd^{\mathrm{conv}}_{\RP^2}(S)$ is connected, hence the image of $\Hol$ lies in the Hitchin component $\Hit(S,3)$. Choi--Goldman~\cite{ChoiGoldmanRP2} proved that $\Hol$ gives a homeomorphism between $\Dd^{\mathrm{conv}}_{\RP^2}(S) $ and $\Hit(S,3)$. This gives a nice geometric interpretation of the Hitchin component as the parameter space of convex $\RP^2$-structures on the surface.

In Baraglia's thesis~\cite{BaragliaThesis}, he shows how to see these convex $\RP^2$-structures using Higgs bundles. Every $\rho\in \Hit(S,3)$ admits a lift to a representation $\bar{\rho}\colon \pi_1(S)\ra \sSL(3,\R)$. By a theorem of Loftin~\cite{AffSpheresConvexRPn} and Labourie~\cite{LabourieCubic}, there exists a complex structure $\Sigma$ and a cubic differential $q_3 \in H^0\big(\Sigma,K^3\big)$ such that the representation $\bar{\rho}$ is the monodromy of the flat connection of the Higgs bundle $(E,\varphi)$, where
\begin{gather*} E = K \oplus \Oo \oplus K^{-1}, \qquad \varphi =
\begin{pmatrix}
0 & 0 & q_3\\
1 & 0 & 0\\
0 & 1 & 0
\end{pmatrix}. \end{gather*}

Baraglia \cite{BaragliaThesis} proved that the solution $H$ of Hitchin's equations for this Higgs bundle is diagonal:
\begin{gather*}H = \begin{pmatrix}
h^{-1} & 0 & 0\\
0 & 1 & 0\\
0 & 0 & h
\end{pmatrix}.\end{gather*}

To construct an $\RP^2$-structure on $\Sigma$, we can choose the section given by the line subbundle $\Oo$. We can verify in the usual way that this section is transverse, hence it gives an $\RP^2$-structure, which can be checked to be convex.

When $q_3=0$, the representation takes values in $\sSO(1,2)$, and the convex set is precisely an ellipsoid, the Klein model of the hyperbolic plane.

\section{Higher-dimensional manifolds} \label{sec:higher dimension}

One limitation of the method described in the previous section is that Higgs bundles can only describe flat bundles on surfaces. We would like to apply similar methods to construct geometric structures on higher-dimensional manifolds, but we need to find a good way to describe the flat bundle. This is possible in some special cases, when the representation factors through a surface group.

\subsection{Sections of the holonomy map} \label{subsec:geometric interpretation of characters}

Let $N$ be a closed manifold and $\sG$ be a reductive Lie group. Consider the character variety
\begin{gather*}\Xx(\pi_1(N),\sG). \end{gather*}
Sometimes, it is possible to find special open subsets $\Uu \subset \Xx(\pi_1(N),\sG)$ which parametrize geometric structures on $N$. To give a meaning to this, we first need to find a manifold $X$ with a transitive and effective action of $\sG$ and $\dim(X)=\dim(N)$. Consider then the holonomy map
\begin{gather*}\Hol\colon \ \Dd^*_{(X,\sG)}(N) \ra \Xx(\pi_1(N),\sG).\end{gather*}
We want to find an open subset $\Uu \subset \Xx(\pi_1(N),\sG)$ and a map
\begin{gather*}T\colon \ \Uu \ra \Dd^*_{(X,\sG)}(N)\end{gather*}
such that $\Hol \circ T = \Id_\Uu$. Such a map $T$ is a section of the holonomy map on $\Uu$. Finding such a $T$ gives a geometric interpretation to the open subset $\Uu$: it becomes a parameter space for a~special subset of $(X,\sG)$-structures on $N$.

\begin{Example}In the previous section, we have seen some very interesting examples of this construction:
\begin{gather*}
\Xx(\pi_1(S), \sPGL(2,\R)) \supset \Hit(S,2)\ra \Dd_{\bH^2}(S) = \Tt(S),\\
\Xx(\pi_1(S), \sPGL(2,\C)) \supset \QFuch(S) \ra \Dd_{\CP^1}(S),\\
\Xx(\pi_1(S), \sPGL(3,\R)) \supset \Hit(S,3)\ra \Dd^{\mathrm{conv}}_{\RP^2}(S) \subset \Dd_{\RP^2}(S).
\end{gather*}
\end{Example}

If we want to find more examples like these, the hypothesis that $\dim(X)=\dim(N)$ becomes a~serious problem: for some groups $\sG$ we don't have homogeneous spaces of the correct dimension. To relax this condition, we will look for geometric structures on another closed manifold $M$. At this point, we don't even need that $N$ is a manifold: the role of $\pi_1(N)$ will be played by a finitely generated group $\Gamma$. Consider the character variety
\begin{gather*}\Xx(\Gamma,\sG). \end{gather*}

We want to use an open subset of it to parametrize $(X,\sG)$-structures on a~closed manifold $M$ (with $\dim(M)=\dim(X)$) which is related with $\Gamma$ by a~group homomorphism $\alpha\colon \pi_1(M) \ra \Gamma$. This group homomorphism induces a map
\begin{gather*}\alpha^*\colon \ \Xx(\Gamma,\sG) \ni \rho \ra \rho \circ \alpha \in \Xx(\pi_1(M),\sG).\end{gather*}
We want to find an open subset $\Uu \subset \Xx(\Gamma,\sG)$ and a map
\begin{gather*}T\colon \ \Uu \ra \Dd^*_{(X,\sG)}(M)\end{gather*}
such that $\Hol \circ T = \alpha^*|_\Uu$. Finding such a map $T$ gives a geometric interpretation to the open subset $\Uu$ as a parameter space for a special subset of $(X,\sG)$-structures on $M$.

Many examples of this scenario come from the theory of domains of discontinuity for Anosov representation in geometries of parabolic type (see the discussion at the end of Section~\ref{subsec:morphisms}, Guichard and Wienhard \cite{GWDomainsofDiscont} and Kapovich, Leeb and Porti \cite{KLPAnosov1}). Assume that $\sG$ is semi-simple, $\sP \subset \sG$ is a parabolic subgroup, $\Gamma$ is Gromov-hyperbolic and torsion-free, $\Uu$ is a connected component of $\sP\text{-}\Anosov(\pi_1(S),\sG)$. Then, we need to choose a geometry $(X,\sG)$ of parabolic type which is in a special relation with $\sP$, in a way that the theory of domains of discontinuity guarantees the existence of a co-compact domain of discontinuity $\Omega_\rho \subset X$ for all the $\sP$-Anosov representations. Under these hypotheses, the topology of the manifold $M=\Omega_\rho/\rho(\Gamma)$ does not depend on $\rho$, and the map
\begin{gather*}T\colon \ \Uu \ni \rho \ra \Omega_\rho/\rho(\Gamma) \in \Dd^*_{(X,\sG)}(M) \end{gather*}
has all the properties listed above. In this way, we give a geometric interpretation to many open connected subsets of Anosov representations as parametrizing $(X,\sG)$-structures on a closed manifold $M$. For an example where these hypotheses are satisfied, see Section~\ref{subsec:dod}.

Even if we know the group $\Gamma$, we usually have no idea what the topology of $M$ is:

\begin{Question} \label{question:topology of M} For some connected component $\Uu$ of $\sP\text{-}\Anosov(\pi_1(S),\sG)$, understand the map
\begin{gather*}T\colon \ \Uu \ra \Dd^*_{(X,\sG)}(M). \end{gather*}
The first step is to determine the topology of $M$.
\end{Question}

\subsection{Transverse maps and transverse submanifolds}

Let $\Gamma$ be a finitely generated group, and $\rho\colon \Gamma \ra \sG$ a representation.

\begin{Definition}Let $M$ be a manifold. A representation $\bar{\rho}\colon \pi_1(M) \ra \sG$ \defin{factors through $\rho$} if there exists a group homomorphism $\alpha\colon \pi_1(M) \ra \Gamma$ such that $\bar{\rho} = \rho \circ \alpha$.
\end{Definition}

Assume now that $\Gamma = \pi_1(N)$ for some manifold $N$. If $N$ is aspherical (for example, if $N$ is a surface), then every group homomorphism $\alpha\colon \pi_1(M) \ra \pi_1(N)$ is induced by a smooth map $f\colon M \ra N$, i.e., $\alpha = f_*$. If $N$ is not aspherical, this is not automatic.

\begin{Definition}
Let $\rho\colon \pi_1(N) \ra \sG$ be a representation. A representation $\bar{\rho}\colon \pi_1(M) \ra \sG$ \defin{strongly factors through $\rho$} if there exists a smooth map $f\colon M \ra N$ such that $\bar{\rho} = \rho \circ f_*$.
\end{Definition}

In this case, the representation $\bar{\rho}$ is the monodromy of a flat bundle $\bar{p}\colon \bar{B} \ra M$, with fiber $X$, and $\rho$ is the monodromy of a flat bundle $p\colon B\ra N$, with fiber $X$, where $(X,\sG)$ is a geometry. The former bundle is isomorphic to the pull-back of the latter by the map~$f$:
\begin{gather*} \bar{B} = f^* B.\end{gather*}

Consider the following commutative diagram:
\begin{gather*}\begin{matrix}
 \bar{B} & \stackrel{f_+}{\lra} & B \\
\scriptstyle{\bar{p}}\Big\downarrow \ & & \scriptstyle{p}\Big\downarrow \ \\
 M & \stackrel{f}{\lra} & N.
\end{matrix} \end{gather*}

\begin{Proposition}
There is a homeomorphism between $\Gamma(M,\bar{B})$, the space of smooth sections of $\bar{B}$, and the space of smooth functions $s\colon M \ra B$ satisfying $p \circ s = f$, endowed with the $C^\infty$ topology. The homeomorphism is given by
\begin{gather*}\Gamma(M,\bar{B}) \ni \bar{s} \ra s = \bar{s} \circ f_+ \in C^\infty(M,B), \end{gather*}
where $f_+\colon \bar{B} \ra B$ is the map given by the pull-back.
\end{Proposition}

Now let's change perspective and think that we don't know the map $f\colon M\ra N$ in advance. Let's just start from a map $s\colon M \ra B$. From $s$, we can construct a~map $f = p \circ s\colon M \ra N$, a~representation $\bar{\rho} = \rho \circ f_*$, a flat bundle $\bar{B} = f^* B$ and a section $\bar{s} \in \Gamma(M,f^* B)$. We will call the section $\bar{s}$ the \defin{tautological section}, because it is just a~reinterpretation of the map $s$ as section of a bundle.

\begin{Definition}We will say that a smooth map $M \ra B$ is a \defin{transverse map} if it is transverse to the parallel foliation of the flat bundle $B$.
\end{Definition}

\begin{Proposition}The map $s$ is a transverse map if and only if the tautological section $\bar{s} \in \Gamma(M,f^* B)$ is a transverse section.
\end{Proposition}

We can summarize the constructions in this subsection with the following proposition.

\begin{Proposition}Let $\rho\colon \pi_1(N) \ra \sG$, and $B$ be a flat bundle with holonomy $\rho$ and fiber $X$, where $(X,\sG)$ is a geometry. Every $(X,\sG)$-structure on some manifold $M$ with holonomy that strongly factors through $\rho$ comes from a transverse map $M \ra B$.
\end{Proposition}

From this proposition we see that we can construct geometric structures on a manifold of higher dimension than $N$, by only understanding the parallel foliation of a flat bundle over $N$.

Recall that, when $\dim(M) = \dim(X)$, transverse maps are always immersions. An interesting special case is when the transverse map is an embedding. In that case we can confuse the map with its image, a submanifold of $B$. Even more interesting is the case when the submanifold is a subbundle of $B$ (not necessarily with structure group $\sG$).

\begin{Definition}A \defin{transverse submanifold of $B$} is a submanifold that is transverse to the parallel foliation of $B$. A \defin{transverse subbundle} is a subbundle that is a transverse submanifold.
\end{Definition}

Transverse submanifolds and transverse subbundles of $B$ can be constructed without any a~priori knowledge of $M$, for example as zero loci of systems of equations defined on $B$. In the case of a transverse subbundle, we can hope to get an explicit description of its topology.

This discussion suggests a reformulation of Question~\ref{question:topology of M}:

\begin{Question}
With the notation of Question \ref{question:topology of M}, add the hypothesis that $\Gamma=\pi_1(N)$ is the fundamental group of a~closed aspherical manifold $N$. Is it true that $M$ is homeomorphic to a~fiber bundle over~$N$?
\end{Question}

This question generalizes a conjecture by Dumas and Sanders:

\begin{Conjecture}[Dumas--Sanders \cite{DumasSanders}] \label{conj:dumas sanders}Let $\sG$ be a simple complex Lie group and $\rho\colon \pi_1(S)\ra \sG$ be a quasi-Hitchin representation. Consider a geometry $(X,\sG)$ of parabolic type, and assume that~$\rho$ has a co-compact domain of discontinuity $\Omega \in X$ coming from the construction of Kapovich, Leeb and Porti~{\rm \cite{KLPAnosov1}}. Then the manifold $\Omega/\rho(\pi_1(S))$ admits a continuous fiber bundle map to the surface~$S$.
\end{Conjecture}

We can prove this conjecture in some special cases, see Section~\ref{subsec:dod}.

In the following we will assume that $N$ is a closed surface. In this case we can use Higgs bundles to describe the bundle $B$ and construct geometric structures on higher-dimensional manifolds by finding transverse subbundles of~$B$.

\section{Geometric structures on circle bundles over surfaces} \label{sec:circle bundles}

We will now show examples where we can construct transverse subbundles which are $3$-dimen\-sio\-nal manifolds. In this case they are circle bundles over surfaces. The case of manifolds of dimension higher than $3$ is more complicated, and it will be treated in the next section.

\subsection{Convex-foliated real projective structures}

Let's consider the Hitchin component
\begin{gather*}\Hit(S,4) \subset \Xx(\pi_1(S),\sPGL(4,\R)).\end{gather*}
We can here see an example of the ideas explained in Section~\ref{subsec:geometric interpretation of characters} about how to interpret this component as parameter space for a special subset of $\RP^3$-structures on $T^1 S$, the unit tangent bundle of the surface.

Guichard--Wienhard \cite{convexfoliatedprojective} proved that every $\rho\in\Hit(S,4)$ has a co-compact domain of discontinuity $\Omega_\rho \subset \RP^3$ which has two connected components $\Omega_\rho = \Omega^+_\rho \cup \Omega^-_\rho$. One of the two (say $\Omega^+_\rho$) has the property that the quotient $\Omega_\rho^+/\rho(\pi_1(S))$ is homeomorphic to $T^1 S$, and that its $\RP^3$-structure is of a special type, called a convex foliated $\RP^3$-structure.

This gives a map
\begin{gather*}\Hit(S,4) \ra \Dd_{(X,\sG)}\big(T^1 S\big) \end{gather*}
that they prove to be a homeomorphism onto a connected component of $\Dd_{(X,\sG)}\big(T^1 S\big) $ containing exactly all convex foliated $\RP^3$-structures on $T^1 S$ (\cite{convexfoliatedprojective}).

We can construct some of these structures using Higgs bundles. In this way we can see some properties of these structures that where not known before.

Every $\rho\in \Hit(S,4)$ admits a lift to a representation $\bar{\rho}\colon \pi_1(S)\ra \sSL(4,\R)$. By a theorem of Labourie \cite{LabourieEnergy}, there exists a complex structure $\Sigma$, a square root $K^{\frac{1}{2}}$ of the canonical bundle and differentials $q_3$, $q_4$ with $q_3 \in H^0\big(\Sigma,K^3\big)$ and $q_4 \in H^0\big(\Sigma,K^4\big)$ such that the representation $\bar{\rho}$ is the monodromy of the flat connection of the following Higgs bundle:
\begin{gather*} E = K^{\frac{3}{2}} \oplus K^{\frac{1}{2}} \oplus K^{-\frac{1}{2}} \oplus K^{-\frac{3}{2}}, \qquad \varphi =
\begin{pmatrix}
0 & 0 & q_3 & q_4\\
1 & 0 & 0 & q_3\\
0 & 1 & 0 & 0\\
0 & 0 & 1 & 0
\end{pmatrix}. \end{gather*}

In Baraglia's thesis \cite{BaragliaThesis}, he considered the special case when the image of $\rho$ is contained in $\sPSp(4,\R)$. In terms of Higgs bundles, this corresponds to the case when $q_3=0$. This belongs to a special type of Higgs bundles called cyclic Higgs bundles, and Baraglia proved that, in this case, the solution $H$ of Hitchin's equations is diagonal
\begin{gather*}H = \begin{pmatrix}
h_1 & 0 & 0 & 0\\
0 & h_2 & 0 & 0\\
0 & 0 & h_3 & 0\\
0 & 0 & 0 & h_4
\end{pmatrix}.\end{gather*}

We can then write the real structure
\begin{gather*}\tau\colon \ E \ni \begin{pmatrix}
v_1\\v_2\\v_3\\v_4
\end{pmatrix}
\ra \begin{pmatrix}
h_4 \bar{v_4}\\ h_3\bar{v_3}\\ h_2\bar{v_2}\\ h_1 \bar{v_1}
\end{pmatrix} \in E. \end{gather*}

The real locus of $E$ is the real vector bundle
\begin{gather*}\Real(E) = \{v\in E \,|\, \tau(v)= v\}. \end{gather*}

We want to construct $\RP^3$-structures, hence we set $X=\RP^3$. The flat bundle $B$ with monodromy $\rho$ and fiber $X$ is $B = \bP(\Real(E))$.

We now want to find a transverse subbundle. Consider
\begin{gather*}M = \bP\left(\left\{
\begin{pmatrix}
0\\v_2\\v_3\\0
\end{pmatrix} \in B \,\Big|\, v_2 = h_3 \bar{v_3} \right\}\right). \end{gather*}
This is a circle bundle over $\Sigma$, isomorphic to the unit tangent bundle. To check that it is transverse, we will put local coordinates on $M$, using a local holomorphic coordinate $z$ on $\Sigma$, and a real coordinate $\theta$ on the circle fiber. For every $s$ local section, we compute the derivatives
\begin{gather*}\nabla_{\!\!\!\frac{\partial}{\partial z}}s, \qquad \nabla_{\!\!\!\frac{\partial}{\partial \bar{z}}}s, \qquad \nabla_{\!\!\!\frac{\partial}{\partial \theta}}s. \end{gather*}
We can then check the transversality condition:
\begin{gather*}\forall\, A\in \C, \ \forall\, B,C\in \R, \qquad A \nabla_{\!\!\!\frac{\partial}{\partial z}}s + \bar{A} \nabla_{\!\!\!\frac{\partial}{\partial \bar{z}}}s + B \nabla_{\!\!\!\frac{\partial}{\partial \theta}}s + C s = 0 \qquad \Rightarrow \qquad A=B=C=0. \end{gather*}
After the transversality condition has been verified, Baraglia also proves that these $\RP^3$-struc\-tu\-res are convex foliated. In this way, he obtains the following new result:

\begin{Theorem}[Baraglia \cite{BaragliaThesis}]For every convex foliated $\RP^3$-structure on $T^1 S$ with holonomy in $\sPSp(4,\R)$, there exists a map $T^1 S \ra S$ which is a circle bundle and has the property that every circle fiber is a projective line for the $\RP^3$-structure.
\end{Theorem}

This statement is completely geometric, but there is no known geometric proof of it, the only proof is the one with Higgs bundles. We conjecture that the same statement is true in general, for all convex foliated $\RP^3$-structures also when the holonomy is not restricted to be in~$\sPSp(4,\R)$.

Working with Qiongling Li, we explored and modified this construction. With the same Higgs field as before, we noticed that the other choice of subbundle
\begin{gather*}M' = \bP\left( \left\{
\begin{pmatrix}
v_1\\0\\0\\v_4
\end{pmatrix} \in B \,\Big|\, v_1 = h_4 \bar{v_4} \right\}\right) \end{gather*}
is also a transverse subbundle, and gives the other $\RP^3$-structure we discussed before, the one given by $\Omega_\rho^-/\rho(\pi_1(S))$. This is a circle bundle over $S$ with Euler class $6g-6$.

We then changed the Higgs field, considering the case when $q_4=0$, but $q_3\neq 0$. For this kind of Higgs bundles, the solution of Hitchin's equations are again diagonal (see Collier--Li~\cite{CollierLi}), and the construction can be applied in a similar way.

Moreover, since the transversality condition is an open condition, we can also understand the cases when at least one of $q_3$, $q_4$ is small enough:

\begin{Theorem}[Alessandrini--Li, work in progress]Let $\rho \in \Hit(S,4)$ be a representation such that the corresponding Higgs bundle has the form given above, with the additional hypothesis that at least one between $q_3$ and $q_4$ is small enough. Then the two $\RP^3$-structures $M=\Omega^+/\rho(\pi_1(S))$ and $M'=\Omega^-/\rho(\pi_1(S))$ have the property that there exist maps $M \ra S$ and $M' \ra S$ which are circle bundles with the property that every circle fiber is a projective line for the $\RP^3$-structure.
\end{Theorem}

\subsection{Closed anti-de Sitter 3-manifolds} \label{subsec:ads}

Consider a symmetric bilinear form $Q$ on $\R^4$ of signature $(2,2)$. This form is preserved by the group $\sO(2,2)$ which has four connected components. The connected component of the identity is called $\sSO_0(2,2)$.

On $\RP^3$, this form defines the open subset
\begin{gather*}{\rm AdS}^3 = \big\{[v] \in \RP^3 \,|\, Q(v,v) > 0 \big\}. \end{gather*}
Let $\sPO(2,2)$ be the projectivization of $\sO(2,2)$, and $\sPO_0(2,2)$ the connected component of the identity. The geometry $\big({\rm AdS}^3, \sPO(2,2)\big)$ is called the $3$-dimensional \defin{anti-de Sitter geometry}, and it is a geometry of pseudo-Riemannian type, carrying an invariant pseudo-Riemannian metric of signature~$(2,1)$.

We can construct the bilinear form $Q$ in the following special way. Consider the vector space~$\R^2$ with a volume form $\omega$. This volume form is preserved by the group $\sSL(2,\R)$. On the tensor product $\R^4 = \R^2\otimes \R^2$ we have a bilinear form $Q = \omega \otimes \omega$. This bilinear form is symmetric and it has signature $(2,2)$. This construction shows us how $\sSL(2,\R)\times \sSL(2,\R)$ acts on~$\R^4$ preserving~$Q$. This gives a homomorphism
\begin{gather*}\sSL(2,\R) \times \sSL(2,\R) \ra \sSO_0(2,2),\end{gather*}
which induces an isomorphism
\begin{gather*}\sPSL(2,\R) \times \sPSL(2,\R) \ra \sPO_0(2,2). \end{gather*}

We will consider a representation $\rho\colon \pi_1(S)\ra \sPO_0(2,2)$ which can be lifted to a representation $\bar{\rho}\colon \pi_1(S)\ra \sSL(2,\R) \times \sSL(2,\R)$. The corresponding Higgs bundle can be written as the tensor product of two Higgs bundles for $\sSL(2,\R)$: given two Higgs bundles for $\sSL(2,\R)$
\begin{gather*}E_1 = L_1 \oplus L_1^{-1}, \qquad Q_1 = \begin{pmatrix}
0 & 1\\
1 & 0
\end{pmatrix}, \qquad \omega_1 = \frac{i}{\sqrt{2}}\begin{pmatrix}
0 & 1\\
-1 & 0
\end{pmatrix}, \\ \varphi_1 = \begin{pmatrix}
0 & a_1\\
b_1 & 0
\end{pmatrix}, \qquad H_1 = \begin{pmatrix}
h_1^{-1} & 0\\
0 & h_1
\end{pmatrix}, \\
E_2 = L_2 \oplus L_2^{-1}, \qquad Q_2 = \begin{pmatrix}
0 & 1\\
1 & 0
\end{pmatrix}, \qquad \omega_2 = \frac{i}{\sqrt{2}}\begin{pmatrix}
0 & 1\\
-1 & 0
\end{pmatrix}, \\ \varphi_2 = \begin{pmatrix}
0 & a_2\\
b_2 & 0
\end{pmatrix}, \qquad H_2 = \begin{pmatrix}
h_2^{-1} & 0\\
0 & h_2
\end{pmatrix}, \end{gather*}
we can form their tensor product
\begin{gather*} E = E_1 \otimes E_2 = L_1L_2 \oplus L_1L_2^
{-1} \oplus L_1^{-1}L_2 \oplus L_1^{-1}L_2^{-1}, \qquad \varphi = \begin{pmatrix}
0 & a_2 & a_1 & 0 \\
b_2 & 0 & 0 & a_1\\
b_1 & 0 & 0 & a_2\\
0 & b_1 & b_2 & 0
\end{pmatrix}.\end{gather*}
The solutions of Hitchin's equations for this Higgs bundle are given by
\begin{gather*}H = \begin{pmatrix}
h_1^{-1}h_2^{-2} & & & \\
 & h_1^{-1}h_2 & & \\
 & & h_1 h_2^{-1} & \\
 & & & h_1 h_2
\end{pmatrix}. \end{gather*}

We now want to construct an ${\rm AdS}^3$-structure on a $3$-manifold with this holonomy. To do so, we will first construct an $\RP^3$-structure, then we verify that the image of the developing map lies inside ${\rm AdS}^3$ by writing the bilinear form $Q=\omega\otimes \omega$ explicitly. Since the developing map goes to ${\rm AdS}^3$ and the holonomy is in $\sPO(2,2)$, the structure we are constructing is actually an ${\rm AdS}^3$-structure.

We consider the following subbundle:
\begin{gather*}M = \bP\big( \Real\big(L_1L_2 \oplus L_1^{-1}L_2^{-1}\big)\big). \end{gather*}
We then start to verify the transversality condition in the usual way. But we notice that the condition is not always verified.

To state the result we need to recall that the solutions of Hitchin's equations for an $\sSL(2,\R)$-Higgs bundle describe an equivariant harmonic map to the hyperbolic plane. So we have the two harmonic maps
\begin{gather*}f_1,f_2\colon \ \widetilde{\Sigma}\ra \bH^2. \end{gather*}
We will denote by $\widetilde{g_1}$, $\widetilde{g_2}$ the two pull-backs of the hyperbolic metric to $\widetilde{\Sigma}$. These tensors are $\pi_1(\Sigma)$-invariant, hence they define two tensors $g_1$, $g_2$ on~$\Sigma$. These symmetric tensors are called the \defin{pull-back metrics}, even though they are not always Riemannian metrics, they can be degenerate at some points.

\begin{Definition}The $\sSL(2,\R)$-Higgs bundle $(E_1,Q_1,\omega_1,\varphi_1)$ \defin{dominates} $(E_2,Q_2,\omega_2,\varphi_2)$ if
\begin{gather*}g_1 - g_2 > 0, \end{gather*}
i.e., if the symmetric tensor $g_1 - g_2$ is positive definite.
\end{Definition}

\begin{Theorem}[Alessandrini--Li \cite{AdSpaper}] The subbundle $M$ is a transverse subbundle if and only if the $\sSL(2,\R)$-Higgs bundle $(E_1,Q_1,\omega_1,\varphi_1)$ dominates $(E_2,Q_2,\omega_2,\varphi_2)$.
\end{Theorem}

In the theory of anti-de Sitter $3$-manifolds, there exists a necessary and sufficient condition for the representation $\rho$ to be the holonomy of an anti-de Sitter structure on a closed manifold. It was shown by Tholozan~\cite{TholozanDomination} that this condition is equivalent to the existence of a complex structure~$\Sigma$ on~$S$ such that $(E_1,Q_1,\omega_1,\varphi_1)$ dominates $(E_2,Q_2,\omega_2,\varphi_2)$. It follows that with Higgs bundles we can construct all closed anti-de Sitter $3$-manifolds with holonomy that lifts to $\sSL(2,\R) \times \sSL(2,\R)$.

Our main motivation for this work was to use the special parametrization that the Higgs bundle give to the manifold~$M$ to explicitly compute invariants of the anti-de Sitter structure, such as the volume. The computation of the volume of the closed anti-de Sitter $3$-manifolds was an open problem that was solved shortly before us by Tholozan~\cite{TholozanVolume}.

\begin{Theorem}[Tholozan \cite{TholozanVolume}, Alessandrini--Li \cite{AdSpaper}]
The volume of the anti-de Sitter structure constructed on $M$ is
\begin{gather*}\Vol(M) = \pi^2\left|\deg(L_1) + \deg(L_2)\right|. \end{gather*}
\end{Theorem}

\section[Projective structures with Hitchin or quasi-Hitchin holonomies]{Projective structures with Hitchin\\ or quasi-Hitchin holonomies} \label{sec:higher dimensions}

In the previous section we presented examples of constructions of geometric structures on $3$-dimensional manifolds. Now we will see how to apply the method to higher-dimensional manifolds. The general strategy is similar, but the technical details are more complicated.

Hitchin and quasi-Hitchin representations act on odd-dimensional real and complex projective spaces admitting a co-compact domains of discontinuity (Guichard--Wienhard~\cite{GWDomainsofDiscont}). The quotient of this domain is a closed manifold with a projective structure. The holomorphic structure of the Higgs bundles and the solutions of Hitchin's equations help us to construct the same projective structure using Higgs bundles, in this way we can determine the topology of the manifold. This is joint work with Qiongling Li~\cite{ProjectiveStructuresHB}.

Another construction of geometric structures on higher-dimensional manifolds using Higgs bundles was done by Collier, Tholozan and Toulisse~\cite{CollierTholozanToulisse}. They constructed photon structures whose holonomy factors through maximal representations in $\sO(2,n)$. This work was not discussed during the mini-course for lack of time.

\subsection{Construction of real and complex projective structures} \label{subsec:projective structures Higgs}

Consider a representation $\rho\colon \pi_1(S) \ra \sPGL(2n,\R)$ in the Fuchsian locus of the Hitchin component $\Hit(S,2n)$. Recall from Example \ref{exa:special representations} that such a representation is the composition of a Fuchsian representation in $\sPGL(2,\R)$ with the irreducible representation $\sPGL(2,\R) \ra \sPGL(2n,\R)$.
We now want to construct $\RP^{2n-1}$ and $\CP^{2n-1}$-structures with holonomy that factors through~$\rho$. Let's start with the corresponding uniformizing Higgs bundle for $\sSL(2,\R)$:
\begin{gather*}E = K^{\frac{1}{2}} \oplus K^{-\frac{1}{2}}, \qquad Q = \begin{pmatrix}
0 & 1\\
1 & 0
\end{pmatrix}, \qquad \omega = \frac{i}{\sqrt{2}}\begin{pmatrix}
0 & 1\\
-1 & 0
\end{pmatrix}, \\ \varphi = \begin{pmatrix}
0 & 0\\
1 & 0
\end{pmatrix}, \qquad H = \begin{pmatrix}
h^{-1} & 0\\
0 & h
\end{pmatrix}. \end{gather*}

The composition with the irreducible representation corresponds, in terms of Higgs bundles, to the symmetric tensor product
\begin{gather*}S(E) = \Symm^{2n-1}(E). \end{gather*}

To make our formulae more explicit and more readable, we will write them only for $n=3$, but similar formulae work for every $n$
\begin{gather*}S(E) = K^{\frac{5}{2}} \oplus K^{\frac{3}{2}} \oplus K^{\frac{1}{2}} \oplus K^{-\frac{1}{2}} \oplus K^{-\frac{3}{2}} \oplus K^{-\frac{5}{2}}, \qquad S(\varphi) = \begin{pmatrix}
0 & 0 & 0 & 0 & 0 & 0\\
1 & 0 & 0 & 0 & 0 & 0\\
0 & 1 & 0 & 0 & 0 & 0\\
0 & 0 & 1 & 0 & 0 & 0\\
0 & 0 & 0 & 1 & 0 & 0\\
0 & 0 & 0 & 0 & 1 & 0
\end{pmatrix}. \end{gather*}

The solution of Hitchin's equations and the real structure are as usual given by
\begin{gather*}S(H) = \begin{pmatrix}
h_1 & \\
 & \ddots \\
 & & h_6
\end{pmatrix}, \qquad
\tau\colon \ S(E) \ni
\begin{pmatrix}
v_1\\ \vdots \\ v_6
\end{pmatrix} \ra
\begin{pmatrix}
h_6 \bar{v_6}\\ \vdots \\ h_1 \bar{v_1}
\end{pmatrix}\in S(E). \end{gather*}

We will consider the subbundles defined by the following equations:
\begin{gather*}U^\C = \bP\left( \left\{
\begin{pmatrix}
h_1^{-\frac{1}{2}} t_1\\ \vdots \\ h_6^{-\frac{1}{2}} t_6
\end{pmatrix} \, \middle| \, t_1 \bar{t_2} + t_3 \bar{t_4} + t_5 \bar{t_6} = 0 \right\} \right), \qquad
U^\R = U^\C \cap \bP\left(\Real(E) \right).\end{gather*}

Similar formulae define these subbundles for every $n$.

\begin{Theorem}[Alessandrini--Li \cite{ProjectiveStructuresHB}]The subbundle $U^\C$ is a transverse subbundle of $\bP(E)$, hence it supports a $\CP^{2n-1}$-structure whose holonomy factors through~$\rho$.

The subbundle $U^\R$ is a transverse subbundle of $\bP(\Real(E))$, hence it supports an $\RP^{2n-1}$-structure whose holonomy factors through~$\rho$.
\end{Theorem}

This method of constructing projective structures has the merit that we can see explicitly the topology of the manifold that supports the structure. Consider the spaces
\begin{gather*}F^\R = T^1 \RP^{n-1}, \qquad F^\C = \big(T^1 \bS^{2n-1}\big)/\sU(1), \end{gather*}
where $\sU(1)$ acts on the unit sphere in a complex vector space $\bS^{2n-1} \subset \C^n$ by component-wise multiplication by a complex number; this action is then lifted to the unit tangent bundle using the differential. Both spaces carry an action of $\sSO(2)$: on $T^1 \RP^{n-1}$, the action of $\sSO(2)$ is given by the geodesic flow, which is periodic. Similarly, $\sSO(2)$ acts via the geodesic flow on $T^1 \bS^{2n-1}$, and this action commutes with the action of $\sU(1)$, hence it descends to an action on the quotient.

\begin{Theorem}[Alessandrini--Li \cite{ProjectiveStructuresHB}] \label{thm:topology of bundle}
Let $P$ be a principal $\sSO(2)$-bundle with Euler class $2g-2$. Then
\begin{enumerate}\itemsep=0pt
\item[$1)$] $U^\R \simeq P\big(F^\R\big)$, for $n \geq 3$,
\item[$2)$] $U^\C \simeq P\big(F^\C\big)$, for $n \geq 2$.
\end{enumerate}
\end{Theorem}

Since the Euler number completely determines a principal $\sSO(2)$-bundle, this result completely determines the topology of the manifolds $U^\R$ and $U^\C$.

\begin{Remark}We can apply the same technique to other representations, namely the ones which are composition of a~Fuchsian representation in $\sSL(2,\R)$ with the diagonal representation $\sSL(2,\R) \ra \sSL(2n,\R)$. Again we can find transverse subbundles and construct $\RP^{2n-1}$ and $\CP^{2n-1}$-structures on these manifolds~\cite{ProjectiveStructuresHB}. This part is less interesting though, since the same construction can be done in a completely geometric way, without using Higgs bundles at all, thanks to the special geometry of the diagonal representation. The interesting thing about the result for the irreducible representation of $\sSL(2,\R)$ is that it is very hard to see these transverse subbundles using only geometry.
\end{Remark}

\subsection{Domains of discontinuity} \label{subsec:dod}

Let $\K=\R$ or $\C$, $\Gamma$ be a Gromov-hyperbolic group and $\rho\colon \Gamma \ra \sPGL(2n,\K)$ be a representation which is $\sP$-Anosov, where $\sP$ is the stabilizer of an $(n-1)$-dimensional projective subspace. The space $\sG/\sP$ can be identified with the Grassmannian $\Gr\big(n,\K^{2n}\big)$, which parametrizes the $n$-dimensional linear subspaces of $\K^{2n}$. The Anosov property gives us the $\rho$-equivariant map
\begin{gather*}\xi\colon \ \partial_\infty \Gamma \ra \Gr\big(n,\K^{2n}\big).\end{gather*}
Guichard--Wienhard \cite{GWDomainsofDiscont} used this map to construct a co-compact domain of discontinuity for the action of $\rho$ in $\KP^{2n-1}$. We first define the $\rho$-equivariant compact subset
\begin{gather*}K^\K_\rho = \bigcup_{t\in \partial_\infty \Gamma} [\xi(t)] \subset \KP^{2n-1},\end{gather*}
which is the complement of the $\rho$-equivariant open subset
\begin{gather*}\Omega^\K_\rho = \KP^{2n-1} {\setminus} K.\end{gather*}

\begin{Theorem}[Guichard--Wienhard \cite{GWDomainsofDiscont}]If $\rho$ is a $P$-Anosov representation in $\sPGL(2n,\K)$, then $\rho$ acts on $\Omega^\K_\rho$ properly discontinuously and co-compactly.
\end{Theorem}

If $\Gamma$ is torsion-free, we can construct the quotient manifold $M^\K = \Omega^\K_\rho/\rho(\Gamma)$, a closed manifold carrying a $\KP^{2n-1}$-structure. The topology of $M^\K$ is constant when $\rho$ varies in a connected component $\Uu$ of $\sP$-$\Anosov(\Gamma,\sPGL(2n,\K))$. In this way, we get a map
\begin{gather*} T\colon \ \Uu \ra \Dd^*_{\KP^{2n-1}}\big(M^\K\big),\end{gather*}
which gives an example of the construction described in Section~\ref{subsec:geometric interpretation of characters}.

Now let's assume that $\Gamma = \pi_1(S)$ is a surface group, and that the connected component $\Uu$ we have chosen is the Hitchin component $\Hit(S,2n)$ when $\K=\R$, and the space of quasi-Hitchin representations when $\K=\C$. In these cases we can understand the topology of $M^\K$, using our construction with Higgs bundles in the previous section.

\begin{Theorem}[Alessandrini--Li \cite{ProjectiveStructuresHB}] \label{thm:dod}
For $n \leq 63$, $M^\K$ is diffeomorphic to $U^\K$.
\end{Theorem}
\begin{proof}
For a representation $\rho$ in the Fuchsian locus, we constructed a $\KP^{2n-1}$-structure on the manifold $U^\K$. We can explicitly compute the developing map of this structure, and we can prove that this structure is isomorphic to the structure $\Omega^\K_\rho/\rho(\pi_1(S))$ if and only if a certain explicit $n \times n$ matrix is positive definite. We then used the computer to check whether the matrix is actually positive definite. Unfortunately, since with the computer we could only check finitely many values of $n$, we stopped after $n=63$.
\end{proof}

We believe the result to be true for every value of $n$, but we don't have a general proof yet. Together with Theorem \ref{thm:topology of bundle}, this result completely describes the topology of $M^\K$. Moreover, this result describes some interesting properties of the geometry of the projective structure on $\Omega^\K_\rho/\rho(\pi_1(S))$ when $\rho$ is close enough to the Fuchsian locus. In particular, Theorem \ref{thm:dod} proves that, for $n \leq 63$, the manifold $M^\C$, is diffeomorphic to a fiber bundle over the surface: this result proves Conjecture \ref{conj:dumas sanders} by Dumas and Sanders in this special case.

The topology of the manifold $M^\R$ was also studied by Guichard and Wienhard (announced in \cite[Remark 11.4(ii)]{GWDomainsofDiscont}). They also saw that it is diffeomorphic to a fiber bundle over the surface with fiber $F^\R$.

More recent work about Conjecture \ref{conj:dumas sanders} uses different methods, which don't involve Higgs bundles. In a joint work with Qiongling Li \cite{ProjectionsHyperbolicPlane}, we proved the following theorem.

\begin{Theorem}[Alessandrini--Li \cite{ProjectionsHyperbolicPlane}]Let $\Omega$ be the domain of discontinuity described in~{\rm \cite{GWDomainsofDiscont}} of a quasi-Hitchin representation $\rho$, where:
\begin{enumerate}\itemsep=0pt
\item[$1)$] $\rho\colon \pi_1(S) \ra \sPGL(2n,\C)$, and $\Omega \subset \CP^{2n-1}$, or
\item[$2)$] $\rho\colon \pi_1(S) \ra \sPGL(n,\C)$, and $\Omega \subset \mathcal{F}_{1,n-1}$, the partial flag manifold parametrizing flags made of lines and hyperplanes.
\end{enumerate}
Then, for every $n$, the quotient $M = \Omega/\pi_1(S)$ is homeomorphic to the total space of a continuous fiber bundle over $S$.
\end{Theorem}

This theorem proves Conjecture~\ref{conj:dumas sanders}, for infinitely many cases. The method we use for the proof has broader scope than the method we used to prove Theorem~\ref{thm:dod}, but it does not allow us to completely understand the topology of the fiber, nor it gives information about the geometric structures.

In a joint work with Maloni and Wienhard \cite{LagrangianGrassmannian}, we proved Conjecture \ref{conj:dumas sanders} in another case:
\begin{Theorem}[Alessandrini--Maloni--Wienhard]Let $\rho\colon \pi_1(S) \ra \sPSp(4,\C)$ be a quasi-Hitchin representation, and let $\Omega$ be the domain of discontinuity described in~{\rm \cite{GWDomainsofDiscont}} of $\rho$ in $\mathrm{Lag}\big(\C^4\big)$, the Lagrangian Grassmannian of~$\C^4$. Then the quotient $M = \Omega/\pi_1(S)$ is homeomorphic to the total space of a continuous fiber bundle over~$S$.
\end{Theorem}
The work continues giving an explicit description of the fiber. Conjecture~\ref{conj:dumas sanders} is still open in general, and it is the subject of active research.

\subsection*{Acknowledgements}

I am grateful to Qiongling Li for the collaboration that brought many of the results surveyed here, to Steve Bradlow, Brian Collier, John Loftin and Anna Wienhard for interesting discussions about this topic and to the anonymous referees for their useful comments on the first draft of the paper. The mini-course was funded by the UIC NSF RTG grant DMS-1246844, L.P.~Schaposnik's UIC Start up fund, and NSF DMS 1107452, 1107263, 1107367 ``RNMS: GEometric structures And Representation varieties'' (the GEAR Network).

\pdfbookmark[1]{References}{ref}
\LastPageEnding

\end{document}